\documentclass[12pt, reqno]{amsart}
\usepackage[latin1]{inputenc}
\usepackage[italian,english]{babel}
\usepackage{amsmath}
\usepackage{amsfonts}
\usepackage{amssymb,esint}
\usepackage{pifont,bbding}
\usepackage{comment}
\usepackage{mathrsfs}
\usepackage{cancel}

\usepackage{calrsfs}

\theoremstyle{plain}
\newtheorem {theorem}{Theorem}[section]
\newtheorem {lemma}[theorem]{Lemma}
\newtheorem {corollary} [theorem]{Corollary}

\newtheorem {proposition} [theorem]{Proposition}
\theoremstyle{definition}
\newtheorem{definition}[theorem]{Definition}
\newtheorem{remark}[theorem]{Remark}

\theoremstyle{remark}

\newcommand{\Hn}{\mathbb H^n}
\newcommand{\R}{\mathbb R}
\newcommand{\N}{\mathbb N}
\renewcommand{\H}{\mathbb H}

\newcommand{\Ln}{\mathcal{L}^{2n+1}}
\newcommand{\p}{\partial}
\newcommand{\Ex}{\mathrm{Exc}}

\newcommand{\e}{\varepsilon}
\renewcommand{\theta}{\vartheta}
\newcommand{\la}{\lambda}

\newcommand{\C}{\mathbb C}

\newcommand{\wh}{\widehat}

\newcommand{\wt}{\widetilde}
\newcommand{\s}{\sigma}

\newcommand{\ex}{\Ex}

\newcommand{\barint}
{\rule[.036in]{.12in}{.009in}\kern-.16in \displaystyle\int}

\newcommand{\SQmu}{\mathcal S^{2n+1}}
\newcommand{\hfr}{\partial^\ast\!}
\newcommand{\res}
{\mathop{\hbox{\vrule height 7pt width .5pt depth 0pt \vrule
height .5pt width 6pt depth 0pt}}\nolimits}

\newcommand{\W}{\mathbb W}

\newcommand{\Lrz}{(\Lambda,r)}
\newcommand{\spt}{\mathrm{spt}\:}
\newcommand{\leb}{\mathcal L}
\newcommand{\dive}{{\mathrm{div}}}
\newcommand{\divg}{{\mathrm{div}}\!_g}

\newcommand{\nablaH} {{\nabla\!_H} }

\newcommand{\Si}{\Sigma}
\newcommand{\Haus}{\mathcal H}
\newcommand{\si}{\sigma}
\newcommand{\ep}{\varepsilon}
\newcommand{\pa}{\partial }

\newcommand{\KR}{K}
\newcommand{\ee}{\mathrm{e}}

\usepackage{color}

\usepackage{aurical}
\newcommand{\h} {\text{\large\Fontamici h}}

\long\def\MSC#1\EndMSC{\def\arg{#1}\ifx\arg\empty\relax\else
     {\par\narrower\noindent
     {\small\it 2010 Mathematics Subject Classification.} \small #1\par}\fi}

\long\def\KEY#1\EndKEY{\def\arg{#1}\ifx\arg\empty\relax\else
     {\par\narrower\noindent
     {\small\it Keywords and Phrases.} \small #1\par}\fi}

\author[Monti]{Roberto Monti}
\address[Monti and Vittone]{Universit\`a di Padova, Dipartimento di Matematica,
via Trieste 63, 35121 Padova, Italy}
\email{monti@math.unipd.it}

\author[Vittone]{Davide Vittone}
\address[Vittone]{ Institut f\"ur Mathematik, Universit\"at Z\"urich, Winterthurerstrasse 190, CH-8057 Z\"urich, Switzerland}
\email{vittone@math.unipd.it}

\begin{document}

\date{\today}

\title[Height estimate and slicing formulas]
{Height estimate and slicing formulas\\
in the Heisenberg group}

\begin{abstract}
We prove a height-estimate (distance from the tangent hyperplane) for $\Lambda$-minima of the perimeter in the sub-Riemannian Heisenberg group. The estimate is in terms of a power of the excess ($L^2$-mean oscillation of the normal) and its proof is based on a new coarea formula for rectifiable sets in the Heisenberg group.
\end{abstract}

\maketitle

\section{Introduction}

In this article, we continue the research 
project started in \cite{MV} and
\cite{M1} on the regularity
of $H$-perimeter minimizing boundaries in the Heisenberg group 
$\H^n$. Our goal is to prove the
so-called {\em height-estimate}
for sets that are {\em $\Lambda$-minima} and have small {\em excess} inside suitable
cylinders, see Theorem \ref{teoaltezza}. 
The  proof follows the scheme of the median choice for the measure of the
boundary in
certain half-cylinders
together with a lower dimensional isoperimetric inequality on slices.
For minimizing currents in $\R^n$, the principal ideas of the argument 
go back to Almgren's paper \cite{A} and are carried over by Federer
in his Theorem 5.3.4 in \cite{F}. The argument can be also found in the Appendix
of  \cite{SS} and,  
for $\Lambda$-minima of perimeter in $\R^n$, in \cite{maggi}.

Our main technical effort is the proof of  
a coarea formula (slicing formula)
for intrinsic rectifiable
sets, see Theorem \ref{teo:rettificabili22}. This formula is established
in Section \ref{two} and has a nontrivial character because the domain of integration
and its slices need not be rectifiable in the standard sense. 
The relative
isoperimetric inequalities that are used
in the slices  reduce  to a single isoperimetric inequality
in  one slice that is relative to a family of varying domains
with uniform isoperimetric constants. This uniformity 
can be established using 
the results on regular domains
in Carnot groups of step 2 of \cite{MM} and the isoperimetric inequality
in \cite{GN}, see
Section \ref{trebis}.

The $2n+1$-dimensional  Heisenberg group  is the manifold
$\H^n=\mathbb C^n
\times\R$, $n\in\N$, endowed with the group product
\begin{equation}
 \label{law!}
   (z,t)\ast (\zeta,\tau) = \big( z+\zeta, t+\tau+ 2 \, \mathrm{Im}\langle   z
, \bar \zeta\rangle \big),
\end{equation}
where $t,\tau\in\R$, $z,\zeta\in \mathbb C^n$ and $\langle   z ,
\bar\zeta\rangle  =
z_1\bar\zeta_1+\ldots+z_n\bar\zeta_n$.
The Lie algebra of left-invariant vector fields in $\H^n$ is
spanned by the vector fields
\begin{equation}
 \label{XY}
  X_j = \frac{\p}{\p x_j}+2y_j\frac{\p}{\p t}, \quad
  Y_j = \frac{\p}{\p y_j}-2x_j\frac{\p}{\p t}, \quad
  \textrm{ and } \quad
  T = \frac{\p}{\p t},
\end{equation}
with $z_j = x_j+i y_j$ and $j=1,\ldots,n$. 
We denote by $H$ the horizontal sub-bundle of $T\Hn$.
Namely, for any $p=(z,t) \in\Hn$ we let $$H_ p = \mathrm{span}
\big\{X_1(p),\ldots,X_n(p), Y_1(p), \ldots, Y_n(p)\big\}.$$
A horizontal section $\varphi\in C^1_c(\Omega; H)$, where 
  $\Omega\subset\H^n$ is an open set,
is
a vector field   of the form
\[
          \varphi = \sum_{j=1}^n \varphi_j X_j+ \varphi_{n+j} Y_j,
\]
where $\varphi_j \in C^1_c( \Omega)$. 

Let $g$ be the  left-invariant Riemannian metric on $\H^n$ 
that makes orthonormal the vector fields $X_1,\dots,Y_n,T$ 
in \eqref{XY}.  For tangent vectors $V,W\in T\Hn$ we let
\[
   \langle V,W\rangle_g = g(V,W)\quad\textrm{and}\quad
   |V|_g = g(V,V)^{1/2}.
\]
The sup-norm 
with
respect to   $g$
of a horizontal section  $\varphi\in C^1_c(\Omega; H)$  is
\[
          \|\varphi\|_g = \max_{p\in\Omega}|\varphi(p)|_g.
\]
The Riemannian   divergence of $\varphi$ is
\[
     \divg \varphi =   \sum_{j=1}^n  X_j\varphi_j+ Y_j \varphi_{n+j} .
\]
The metric $g$ induces a volume form 
on $\Hn$ that is left-invariant. Also the Lebesgue measure
$\Ln$ on $\Hn$ is left-invariant, 
and by the uniqueness of the  Haar measure  
the volume  induced by $g$ is the Lebesgue measure $\Ln$. In fact, the
proportionality
constant is $1$.

The \emph{$H$-perimeter} of a $\leb^{2n+1}$-measurable set
$E\subset\H^n$ in an open set $\Omega\subset\H^n$ is
\[
   \mu_E(\Omega ) 
=\sup\left\{ \int_E \divg \varphi \, d\leb^{2n+1}   : 
\varphi \in C^1_c(\Omega;H),\|\varphi\|_g\leq 1 \right\}.
\]
If $\mu_E(\Omega)<\infty$ we say that $E$ has finite $H$-perimeter in $\Omega$.
If $\mu_E(A)<\infty$ for any open set $A\subset\subset\Omega$,
we say that $E$
has locally finite $H$-perimeter in $\Omega$.
In this case, the open sets mapping $A\mapsto \mu_E(A)$  extends to a Radon
measure $\mu_E$ on $\Omega$ that is called \emph{$H$-perimeter measure} 
induced by $E$.
Moreover, there exists a $\mu_E$-measurable function $\nu_E:\Omega\to H$ such
that $|\nu_E|_g=1$ $\mu_E$-a.e.~and
the Gauss-Green integration by parts formula 
\[
   \int_\Omega \langle \varphi, \nu_E\rangle_g  \, d\mu_E = - \int _\Omega \divg
\varphi\, d\leb^{2n+1}
\]
holds for any $\varphi\in C^1_c(\Omega; H)$.
 The vector $\nu_E$ is called
\emph{horizontal
inner normal} of $E$ in $\Omega$.

The Kor\`anyi 
norm of $p=(z,t)\in \H^n$ is  
$\| p\|_{\KR} =(|z|^4+t^2)^{1/4}$.
For any $r>0$ and $p \in\H^n$, we define the balls
\[
 B_r = \big\{q\in \H^n: \| q\|_\KR<r\big\}\quad\text{and}
\quad
 B_r(p)= \big\{ p\ast q\in\H^n: q\in B_r\big\}.
\]
{The \emph{measure theoretic boundary} of a measurable set $E \subset
\H^n$ is  the set
\[
 \partial E  = \big\{p\in\H^n : \textrm{$\leb^{2n+1} (E\cap B_r(p))>0$ and
$ \leb^{2n+1}(B_r(p)\setminus E)>0$ for all $r>0$}\big\}.
\]
For a set $E$ with locally finite $H$-perimeter, the $H$-perimeter measure
$\mu_E$ is concentrated on $\partial E$ and, actually, on a subset $\partial^\ast E$ of $\partial E$, see below. Moreover, up to modifying $E$ on a Lebesgue negligible set, one can always assume that $\partial E$ coincides with the topological boundary of $E$, see \cite[Proposition 2.5]{SCV}}.

\begin{definition}
       \label{deflambdamin}
Let $\Omega\subset\Hn$ be an open set, 
$\Lambda\in[0,\infty)$, and $r\in
(0,\infty]$. We say that a  set $E\subset\Hn$ with locally finite $H$-perimeter in $\Omega$
is a {\em
$\Lrz$-minimum of $H$-perimeter in $\Omega$} 
if, 
for any measurable set $F\subset\Hn$,   
$p\in\Omega$, and $s<r$
such that   $E\Delta F\subset\subset
B_s(p) \subset \subset  \Omega$,
there holds
\[
 \mu_E (B_s(p))\leq \mu_F (B_s(p)) + \Lambda \Ln(E\Delta F),
\]
where $E\Delta F = E\setminus F\cup F\setminus E$.

We say that  $E$ is {\em locally $H$-perimeter minimizing in $\Omega$} if, for any measurable set $F\subset\Hn$ and any open set $U$ such that $E\Delta F\subset\subset U\subset\subset \Omega$, there holds $ \mu_E (U)\leq \mu_F (U)$.
\end{definition}

\noindent
We will often use the term {\em $\Lambda$-minimum}, rather than 
$\Lrz$-minimum, when the role of $r$ is not relevant.
In Appendix A, we list without proof some elementary properties of
$\Lambda$-minima.

We introduce 
the notion of cylindrical excess.
The \emph{height function} $\h
:\H^n\to\R$ is defined by  $\h(p)=p_1$, where
$p_1$ is the first coordinate  of $p=(p_1,\ldots,p_{2n+1})\in\H^n$.
The set 
$\W=\{p\in \H^n:\h(p)=0\}$
is the vertical hyperplane 
passing through $0\in\H^n$ and orthogonal to the left-invariant
vector field $X_1$.
The disk  in $\W$  of radius $r>0$ centered at $0\in\W$ 
induced by the Kor\`anyi  norm is the set   
$D_r = \big\{p\in \W:\| p\|_{\KR} <r\big\}$.
The intrinsic cylinder with central 
section $D_r$ and height $2r$ is the set
\[
                C_r=D_r\ast (-r,r)\subset\Hn.
\]

\noindent Here and in the sequel, we use
the notation $D_r\ast(-r,r) = 
\{ w\ast (s\ee_1)\in\H^n: w\in D_r,\, s\in(-r,r)\big\}$,
where $s\ee_1= (s,0,\ldots,0)\in\H^n$.
The cylinder $C_r$ is comparable with the ball $B_r=\{\| p\|_\KR<r\}$.
Namely, there exists a constant $k=k(n)\geq 1$ such that
for any $r>0$ we have
\begin{equation}
   \label{eqpallecilindri}
B_{r/k}\subset C_r \subset B_{kr}.
\end{equation}

By a rotation of the system of coordinates, it is enough
to consider excess in cylinders with basis in $\W$ and axis $X_1$.

\begin{definition}[Cylindrical excess]\label{def:cilexc}
Let $E\subset \H^n$ be a set with locally finite $H$-perimeter.
We define the excess of $E$ in the cylinder $C_r$ oriented by the
vector
$\nu= -X_1$ as
\[
\ex (E,r,\nu)=\frac{1}{2 r^{2n+1}}
 \int_{C_r}|\nu_E - \nu |_g^2 \,d\mu_E,
\]
where $\mu_E$ is the $H$-perimeter measure of $E$ and $\nu_E$ is
its horizontal inner normal.
%
\end{definition}

\begin{theorem}[Height estimate]
   \label{teoaltezza}
Let $n\geq 2$. 
There exist constants $\e_0=\e_0(n)>0$ 
and $c_0=c_0(n)>0$ with the following
property. If $E\subset \H^n $ is a $\Lrz$-minimum 
of $H$-perimeter 
in the cylinder $C_{4k^2r}$,
$\Lambda r\leq 1$, $0\in\partial E$, and
\[
\ex (E,4k^2r,\nu)\leq \e_0,
\]
then
\begin{equation}\label{STA}
   \sup\big\{ |\h(p) | \in [0,\infty) : p\in \partial E\cap C_{r}\big\} \leq
c_0\, r\,
\ex (E,4k^2r,\nu)^{\tfrac{1}{2(2n+1)}}.
\end{equation}
The constant $k=k(n)$ is the one in \eqref{eqpallecilindri}.
\end{theorem} 

The estimate \eqref{STA} does not hold when
$n=1$. In fact,  there are sets $E\subset \H^1$ such that 
$\ex(E,C_r,\nu) = 0$ but 
$\partial E$ is not flat in $C_{\varepsilon r}$ for any $\varepsilon >0$.
See the conclusions of Proposition 3.7 in \cite{M1}. 
Theorem \ref{teoaltezza} is proved in Section \ref{tre}.

Besides local minimizers of 
$H$-perimeter, our interest in $\Lambda$-minima 
is also motivated by possible applications to isoperimetric sets. 
The height estimate is a first step in the regularity 
theory of $\Lambda$-minima of classical perimeter; 
we refer to \cite[Part III]{maggi} for a detailed account on the subject.
\medskip

In order to state the slicing formula in its general form, we need
the definition of a rectifible set in $\H^n$ of codimension $1$.
We follow closely \cite{FSSCMathAnn}, where this notion was first introduced.

The Riemannian and horizontal gradients of a function $f\in C^1(\Hn)$ are,
respectively, 
\[
\begin{split}
& \nabla f=(X_1 f)X_1 + \dots + (Y_nf)Y_n + (Tf)T, \\
& \nablaH f=(X_1 f)X_1 + \dots + (Y_nf)Y_n.
\end{split}
\]
We say that  a continuous function $f\in C(\Omega)$, with $\Omega\subset
\Hn$ open set,
is of  class $C^1_H(\Omega)$ if
the horizontal gradient $\nablaH f$ exists in the sense of distributions 
and is represented by continuous functions $X_1 f, \ldots, Y_nf$ in $\Omega$.
A set $S\subset\H^n$ is an $H$-regular hypersurface if for all $p\in S$ there
exist  $r>0$ and a function $f\in C^1_H( B_r(p))$ such that
$S\cap B_r(p) = \big\{ q\in B_r(p): f(q) = 0 \big\}$ and $\nablaH f(p)\neq 0$. Sets with $H$-regular boundary have locally finite $H$-perimeter.

For any $p=(z,t)\in\H^n$, let us define the box-norm $
 \| p \|_\infty  = \max \{ |z|, |t|^{1/2}\} $ and the balls
$U_r = \{q\in \H^n:\|q\|_\infty <r\}$ and $U_r(p) = p\ast U_r$, with $r>0$.
Let $E\subset \H^n$ be a set.
For any $s\geq0$ define the measure
\[ 
 \mathcal S ^s (E)
 =  \sup_{\delta>0} \ 
    \inf \Big\{ c(n,s) \sum_{i\in\N} r_i^s   : E\subset \bigcup_{i\in\N}  
U_{r_i}(p_i), \, r_i <\delta \Big\}.
\]
Above, $c(n,s)>0$ is a normalization constant 
that we do not need to specify,
here.
By Carath\`eodory's construction,  $E\mapsto
\mathcal S^s(E) $ is a Borel measure in $\H^n$.
When $s=2n+2$, $\mathcal S^{2n+2}$ turns out to be the Lebesgue measure
$\mathcal L^{2n+1}$.
Thus, the correct dimension to measure hypersurfaces is $s=2n+1$.
In fact, if $E$ is a set with locally finite $H$-perimeter in $\H^n$, then
we have
\begin{equation}\label{missa}
      \mu_E = \mathcal S^{2n+1} \res \partial^* E,
\end{equation}
where $\res$ denotes restriction and 
$\partial^*E$ is the $H$-reduced boundary of $E$, 
namely the set of points $p\in \Hn$ such that 
$\mu_E(U_r(p))>0$ for all $r>0$,
$\fint_{U_r(p)} \nu_E \, d\mu_E\to \nu_E(p)$ 
as $r\to0$ and $|\nu_E(p)|_g=1$.
The validity of formula \eqref{missa} depends on the geometry of the balls 
$U_r(p)$, see \cite{Ma4}. We refer the reader to 
  \cite{FSSCMathAnn} for more details on the $H$-reduced boundary.

\begin{definition}
\label{REC}
 A set $R \subset \H^n$ is $\mathcal S^{2n+1}$-rectifiable if there
exists
 a sequence of $H$-regular hypersurfaces $(S_j)_{j\in\N}$ in $\H^n$
such that
\[
 \mathcal S^{2n+1} \Big (R\setminus \bigcup_{j\in\N} S_j\Big) = 0.
\]
\end{definition}

By the results of \cite{FSSCMathAnn}, 
the $H$-reduced boundary $\partial^*E$
is  $\mathcal S^{2n+1}$-rectifiable. 
Definition \ref{REC}
is generalized in \cite{MSSC}, where the authors
study the notion of an $s$-rectifiable set in $\H^n$ 
for any integer 
$1\leq s\leq 2n+1$.

An $H$-regular surface $S$ has a continuous horizontal normal $\nu_S$ that is
locally defined up to the sign. This normal is given by the formula
\begin{equation}\label{nuesse}
 \nu_S = \frac{\nablaH f}{|\nablaH f|_g},
\end{equation}
where $f$ is a defining function for $S$. When $S=\partial E$ is the boundary
of a smooth set, then $\nu_S$ agrees with the horizontal normal $\nu_E$. 
Then, for an
 $\mathcal S^{2n+1}$-rectifiable 
set $R \subset \H^n$ there is a unit horizontal normal $\nu_R: R\to H$
that is Borel regular.
This normal is uniquely defined 
$\mathcal S^{2n+1}$-a.e.~on $R$
up the the sign, see Appendix B. However, formula \eqref{muchofrio22} below
does not depend on the  sign.

In the following theorem, $\Omega\subset\H^n$ is an open set and $u\in
C^\infty(\Omega)$ is
a smooth function. For any $s\in\R$, we denote by $\Sigma^s = \big\{
p\in\Omega: u(p) =s\big\}$ the level sets of $u$.

\begin{theorem}
   \label{teo:rettificabili22}
Let $R\subset\Omega$ be an 
$\mathcal S^{2n+1}$-rectifiable set.
Then, for a.e.~$s\in\R$ there
exists a Radon measure $\mu_R^{s}$ on  
$R\cap\Si^s$ such that for any Borel function
$h:\Omega\to[0,\infty)$ the function
\begin{equation}\label{filippo}
 s\mapsto \int_{\Omega}h\,
 \frac{|\nablaH u|_g}{|\nabla u|_g}\,d\mu_R^{s}
\end{equation}
is $\mathcal L^1$-measurable, and
 we have the coarea formula
\begin{equation}\label{muchofrio22}
\int_\R \int_{\Omega}h\,\frac{|\nablaH u|_g}
 {|\nabla u|_g}\,d\mu_R^{s}\, ds
= \int_{R}h\,\sqrt{|\nablaH u|_g^2-\langle \nu_R,\nablaH
u\rangle_g^2}\,d\SQmu.
\end{equation}
\end{theorem}

\noindent
Theorem \ref{teo:rettificabili22}
is proved in Section \ref{two}.
When $R\cap \Sigma^s$ is a regular subset of $\Sigma^s$, the measures
$\mu_R^s$ are natural horizontal perimeters defined in $\Sigma^s$.

Coarea formulae in the Heisenberg group are known only 
for slicing of sets with positive Lebesgue measure, 
see \cite{Ma1,Ma3}. Theorem \ref{teo:rettificabili22} 
is, to our knowledge, the first example of slicing 
of lower-dimensional sets in a sub-Riemannian framework. 
Also, Theorem \ref{teo:rettificabili22} is a nontrivial extension
of the Riemannian coarea formula, because the set $R$ and the slices
$R\cap \Sigma^s$ need not be rectifiable in the standard sense, 
see \cite{KSC}. We need the coarea formula 
\eqref{muchofrio22} in the proof of Theorem \ref{teoaltezza},
see Section \ref{pollo}.
 
We conclude the introduction by stating a different but equivalent formulation of the coarea formula \eqref{muchofrio22} that is closer to standard coarea formulae. This alternative formulation holds only when $n\geq 2$: when $n=1$, the right hand side in \eqref{eq:coarealiscia2a} might not be well defined, see Remark \ref{rem:notrascH1}.

\begin{theorem}\label{teo:coarealiscia2}
Let $\Omega\subset\Hn$, $n\geq 2$, be an open set,
$u\in C^\infty(\Omega)$ be a smooth function, and
$R\subset\Omega$ be an $\mathcal S^{2n+1}$-rectifiable set. 
Then, for any Borel
function $h:\Omega\to[0,\infty)$ there holds
\begin{equation}
     \label{eq:coarealiscia2a}
\int_\R \int_\Omega h\,d\mu_R^{s}\,ds    
     =   
\int_{R} h\, |\nabla u|_g\sqrt{1-\left\langle \nu_R,
\tfrac{\nablaH u}{|\nablaH u|_g}\right\rangle_g^2}\,d\SQmu,
\end{equation}
where $\mu_R^s$ are the measures given by Theorem
\ref{teo:rettificabili22}.

\end{theorem}

\section{Proof of the coarea formula}
\label{two}

\subsection{Horizontal perimeter on submanifolds}

Let $\Si\subset\Hn$ be a  $C^\infty$ hypersurface. 
We define the horizontal tangent bundle $H\Si$ letting,
for any $p\in\Si$,
\[
 H_p \Si=H_p \cap T_p \Si.
\]
In general, the rank of $H\Si$ is not constant.
This depends on the  presence of {\em characteristic points} on $\Si$, i.e.,
points such that $H_p = T_p\Si$. 
For points $p\in \Si$ such that $H_p\neq T_p\Si$, we have 
$\mathrm{dim}(H_p\Si)=2n-1$.

We denote by $\si_\Si$ the surface measure on $\Si$ induced 
by the Riemannian metric $g$ restricted to the
tangent bundle $T\Si$.

\begin{definition} \label{DEF_SIGMA}
Let $F\subset \Si$ be a Borel set and 
let $\Omega\subset\Si$ be an  open set.
We define the {\em $H$-perimeter of $F$ in $\Omega$}
\begin{equation}\label{SUPDEF}
\mu_F^\Si(\Omega)=\sup\left\{ \int_F \divg \varphi \:d\si_\Si : 
\varphi \in C^1_c(\Omega;H\Si),\|\varphi\|_g\leq 1 \right\}.
\end{equation}
We say that the set $F\subset \Si$
has locally finite $H$-perimeter in $\Omega$
if  $\mu_F^\Si(A)<\infty$
for any  open set $A\subset\subset  \Omega $.
\end{definition}

\medskip

By Riesz' theorem, if 
$F\subset\Si$ has locally finite $H$-perimeter in $\Omega$,
then the open set mapping $A \mapsto \mu_F^\Si(A)$ 
extends to a Radon measure on $\Omega$, called \emph{$H$-perimeter measure}
of $F$. 


\begin{remark}\label{rem:tuttoliscio}
If $F\subset\Si$ is an open set with smooth boundary, 
then by the divergence theorem we have, 
for any $\varphi \in C^1_c(\Omega;H\Si)$,
\begin{equation}\label{smotto}
      \int_F \divg \varphi\:d\si_\Si =
    \int_{\partial F} \langle N_{\partial F},\varphi\rangle_g\:
      d\lambda_{\partial F} ,
\end{equation}
where $N_{\partial F}$ is the Riemannian outer unit normal 
to $\partial F$ and $d\lambda_{\partial F}$ is the 
Riemannian $(2n-1)$-dimensional volume form on $\partial F$
induced by $g$.

From the sup-definition \eqref{SUPDEF} and from \eqref{smotto},
we deduce that the $H$-perimeter measure of $F$ has the following representation
\[
     \mu_F^\Si = |N_{\partial F} ^{H\Si}|_g \,\lambda_{\partial F},
\]
where $N_{\partial F}^{H\Si}\in H\Si$ is the $g$-orthogonal 
projection
of $N_{\partial F}\in T\Si$ onto $H\Si$.
\end{remark}

This formula can be generalized as follows. 
We denote by $\Haus_g^{2n-1}$ the $(2n-1)$-dimensional 
Hausdorff measure in $\Hn$ induced by the metric $g$.

\begin{lemma}\label{lem:quasiliscio}
Let $F,\Omega\subset\Si$ be open sets and assume   that 
there exists a compact  set $N\subset \partial F$ 
such that $\Haus_g^{2n-1}(N)=0$ and $(\partial F\setminus N)\cap \Omega $ 
is a smooth   $(2n-1)$-dimensional surface. Then, we have
\begin{equation}\label{eq:quasiliscio}
\mu_F^\Si  \res \Omega= |N_{\partial F} ^{H\Si}|_g \,
 \lambda_{\partial F\setminus N}\res \Omega.
\end{equation}
\end{lemma}

\begin{proof}
For any $\ep>0$ there exist points $p_i\in\Hn$ and radii
$r_i\in(0,1)$, $i=1,\ldots, M$, such that
\[
   N \subset\bigcup_{i=1}^M 
     B_g(p_i,r_i)\quad\text{and}\quad\sum_{i=1}^M r_i^{2n-1}<\ep,
\]
where $B_g(p,r)$ denotes the ball in $\Hn$
with center $p$ and radius $r$ with respect to the 
metric $g$. By a  partition-of-the-unity argument, there exist
functions  $f^\ep,g_i^\ep
 \in C^\infty(\Omega;[0,1])$, $i=1,\ldots,M $, such that
\begin{itemize}
\item[i)] $\displaystyle  f^\ep +g_1^\ep+\ldots+g_M^\ep 
  =\chi_\Omega$ ;

\item [ii)] $ f^\ep=0$   on $\bigcup_{i=1}^M B_g(p_i,r_i/2)$; 

\item[iii)] $ \spt g_i^\ep\subset B_g(p_i,r_i)$ for each $i$;

\item[iv)]  $ |\nabla g_i^\ep|_g \leq Cr_i^{-1}$ for a constant 
    $C>0$ independent of $\ep$.
\end{itemize}
Hence, for any horizontal section 
$\varphi\in C^1_c(\Omega;H\Si)$ we have
\begin{equation}
\label{UN}
\begin{split}
\int_F \divg\varphi \, d\si_\Si  = & \int_{F}
\divg(f^\ep\varphi)d\si_\Si 
 + \sum_{i=1}^M \int_{F\cap B_g(p_i,r_i)} \divg(g_i^\ep\varphi)d\si_\Si\\
 = & \int_{\partial F\setminus N} \langle f^\ep \varphi,
 N_{\partial F}\rangle_g d\la_{\partial F\setminus N} 
 + \sum_{i=1}^M  \int_{F\cap B_g(p_i,r_i)} \divg(g_i^\ep\varphi)d\si_\Si,
\end{split}
\end{equation}
where, by iv),  
\begin{equation} \label{DU}
\begin{split}
\left|\sum_{i=1}^M  \int_{F\cap B_g(p_i,r_i)} \divg(g_i^\ep\varphi)
  d\si_\Si\right| \leq& \sum_{i=1}^M 
\int_{B_g(p_i,r_i)} \left( \|\dive _g\varphi\|_{L^\infty} 
 + Cr_i^{-1}\right) d\sigma_\Si\\
\leq &  C' \sum_{i=1}^M  r_i^{2n-1}\ \leq C'\ep,
\end{split}
\end{equation}
with a constant   $C'>0$ independent of $\ep$. 

Letting   $\ep\to0$, we have $f^\ep \to 1$ 
pointwise on $\partial F\setminus N$, by i) and iii).
Then, from \eqref{UN} and \eqref{DU}
we obtain
\[
\int_F \divg\varphi \:d\si_\Si = \int_{\partial F\setminus N} \langle
\varphi,N_{\partial F}\rangle_g d\la_{\partial F\setminus N} 
\]
and  the claim  \eqref{eq:quasiliscio} follows by standard arguments.
\end{proof}

\subsection{Proof of Theorem \ref{teo:rettificabili22}}

Let  $\Omega\subset\Hn$ be an open set and let   $u\in C^\infty(\Omega)$.
By Sard's theorem, for a.e.~$s\in \R$ the level set 
 $$\Si^s=\big\{p\in\Omega:u(p)=s\big\}$$
is a smooth hypersurface and, moreover, we have $\nabla u\neq 0$ 
on $\Sigma^s$.

Let $E\subset\Hn$ 
be a Borel set such that $E\cap \Si^s$ has (locally) 
finite $H$-perimeter in $\Omega\cap \Si^s$,
in the sense of Definition \ref{DEF_SIGMA}.
Then on  $\Omega\cap \Sigma^s$
we have the $H$-perimeter measure 
$\mu ^{\Sigma^s}_{E\cap \Sigma^s}$ induced by $E\cap\Sigma^s$.
We shall use the notation
\[
     \mu^{s}_E = \mu ^{\Sigma^s}_{E\cap \Sigma^s}
\]
to denote a measure on $\Omega$ that is supported
on $\Omega\cap\Sigma^s$.


We start with the following  coarea formula in the smooth case,
that is deduced from the Riemannan formula.

\begin{lemma}\label{lem:coarealiscia}
Let $\Omega\subset\Hn$ be an open set 
and $u\in C^\infty(\Omega)$. Let $E\subset\Hn$ be an open 
set with $C^\infty$ boundary in $\Omega$ such that 
$\mu_E(\Omega)<\infty$. Then we have 
\begin{equation}\label{eq:coarealiscia}
\int_\R \int_\Omega\frac{|\nablaH u|_g}{|\nabla u|_g}\,d\mu_E^{s}\,ds    
  =   
\int_{\Omega}\sqrt{|\nablaH u|_g^2-\langle \nu_E,\nablaH u\rangle_g^2}\,
d\mu_E,
\end{equation}
where $\mu_E$ is the $H$-perimeter measure of $E$ and $\nu_E$
is its  horizontal normal. 
\end{lemma}

\begin{proof}
The integral in the left hand side is well defined,
because for a.e.~$s\in\R$ there holds $\nabla u \neq 0 $ on $\Si^s$.
By the  coarea formula for 
Riemannian manifolds, see e.g.~\cite{BZ}, 
for any  Borel function $h:\partial E\to[0,\infty]$ we have
\begin{equation}\label{eq:coareaclass}
   \int_\R \int_{\pa E\cap \Si^s} h\,d\la_{\pa E\cap \Si^s}\:ds   =   
\int_{\partial E}h\:|\nabla^{\partial E}u|_g\,d\si_{\partial E},
\end{equation}
where $\nabla^{\partial E}u$ is the tangential gradient 
of $u$ on $\partial E$. 
Then we have  
\begin{equation}\label{eq:gradtangusuE}
\nabla^{\partial E}u=\nabla u- \langle \nabla u,N_{\partial E}\rangle_g 
    N_{\partial E}\quad\text{and}\quad 
 |\nabla^{\partial E}u|_g=
 \sqrt{|\nabla u|_g^2 - \langle \nabla u,N_{\partial E}\rangle_g^2}.
\end{equation}

{\em Step 1.} Let us define the set
\[
C=\left\{p\in\pa E\cap\Omega :\nabla u(p)\neq 0 
\text{ and }N_{\pa E}(p)=\pm \frac{\nabla u(p)}{|\nabla u(p)|_g}\right\}.
\]
If $s\in\R$ is such that $\nabla u\neq 0$ on $\Si^s$, 
then $C\cap \Si^s$ is a closed set in $\Si^s$. 
Using the coarea formula 
\eqref{eq:coareaclass} with the function $h=\chi_C$, 
we get
\[
\int_\R \la_{\pa E\cap \Si^s}(C)\,ds = 
\int_{C} |\nabla^{\partial E}u|_g\,d\si_{\partial E} =0 ,
\]
because we have $\nabla^{\partial E}u =0$ on $C$. 
In particular, we deduce that 
\begin{equation}\label{eq:a.e.s}
\text{$C\cap \Si^s$ is a closed set 
 in $\Si^s$\quad and\quad} \la_{\pa E\cap \Si^s}(C\cap \Si^s)=0
\quad \text{for a.e.~$s\in\R$}.
\end{equation}

If $p\in\Si^s$ is a point such that
$\nabla u(p)\neq 0$ and $ p\notin C$,
then  
$\Si^{s}$ is a smooth hypersurface in a neighbourhood of $p$ and
$E^s = E\cap\Si^{s}$ is a domain in $\Si^{s}$ with smooth boundary
in a neighbourhood of $p$. Moreover, we have $(\partial E\cap \Si^s) \setminus C
= 
\partial E^s \setminus C$.
Then, from \eqref{eq:a.e.s} and from Lemma \ref{lem:quasiliscio} 
we conclude that  for a.e.~$s\in\R$ we have
\begin{equation}\label{eq:pppperimetri}
\mu_E^s = |N_{\pa E^s}^{H\Si^s}|_g\la_{\pa E^s}.
\end{equation}
By \eqref{eq:a.e.s} and \eqref{eq:pppperimetri}, there holds
\begin{equation}\label{eq:a.e.s222}
\mu_{E}^{s}(C\cap \Si^s)=\int_{C\cap \Si^s} 
|N_{\pa E^s}^{H\Si^s}|_g   d \la_{\pa E^s}=
0\qquad\text{for a.e. }s\in\R.
\end{equation}

{\em Step 2.} We prove \eqref{eq:coarealiscia} by plugging 
into \eqref{eq:coareaclass}
the Borel function  $h: \pa E\to[0,\infty]$
\[
h=\left\{\begin{array}{ll}
\displaystyle\frac{|N_{\pa E}^{H}|_g
    \sqrt{|\nablaH u|^2_g-\langle\nu_E,\nablaH u\rangle^2_g}}
    {|\nabla u|_g\sqrt{1-\langle N_{\pa E},\frac{\nabla u}{|\nabla u|_g}
      \rangle^2_g}} & \text{on }\pa E\setminus (C\cup \{\nabla u=0\})\\
0 & \text{on }C\cup \{\nabla u=0\}.
\end{array}\right.
\]
Above, $N_{\partial E}^H$ is the projection 
of the Riemannian normal $N_{\partial E}$
onto $H$ and $\nu_E$ is the horizontal normal. Namely, we have
\[
 N_{\pa E}^H= N_{\pa E}-\langle N_{\pa E},T\rangle_g T\quad \text{and}\quad
  \nu_E=\frac{N_{\pa E}^H}{|N_{\pa E}^H |_g}.
\]
The $H$-perimeter measure of $E$ is 
\begin{equation}\label{eq:pavia} 
   \mu_E=|N_{\pa E} ^H |_g \si_{\pa E}.
\end{equation}
Using 
\eqref{eq:gradtangusuE} and 
\eqref{eq:pavia}, we find
\begin{equation}
\label{eq:godot1}
\begin{split}
\int_{\partial E}   h\,
   |\nabla^{\partial E}u|\,d\si_{\partial E} 
& = 
\int_{\partial E\setminus (C\cup \{\nabla u=0\})} 
|N_{\pa E}^H|_g 
 \sqrt{|\nablaH u|_g^2-\langle\nu_E,\nablaH u\rangle_g^2}  
  \, d\si_{\partial E} 
\\
&= \int_{\partial E\setminus (C\cup \{\nabla u=0\})} 
\sqrt{|\nablaH u|^2_g-\langle\nu_E,\nablaH u\rangle_g^2}  
\,d\mu_E
\\
&= \int_{\partial E} \sqrt{|\nablaH u|^2_g-\langle\nu_E,\nablaH u\rangle_g^2} 
 \,
d\mu_E,
\end{split}
\end{equation}
where the last equality is justified by the fact that
if $p\in C\cup \{\nabla u=0\}$ then
$$\sqrt{|\nablaH u(p)|^2_g-\langle\nu_E(p),\nablaH u(p)\rangle_g^2}=0.$$

For a.e.~$s\in\R$, we have $\nabla u\neq 0$ on $\Si^s$.
Using \eqref{eq:pavia} and the fact that  $h=0$ 
on $C\cup \{ \nablaH u=0\} $,
letting $\Lambda ^ s = 
(\partial E\cap \Si^s)\setminus (C\cup \{\nablaH u=0\})$
we obtain 
\begin{equation}\label{godot2}
\begin{split}
\int_\R \int_{\partial E\cap \Si^s} h\,d\la_{\pa E^s}\,ds
& = 
 \int_\R \int_{\Lambda ^s} 
   \frac{|N_{\pa E}^H|_g
   \sqrt{|\nablaH u|^2_g-\langle\nu_E,\nablaH u\rangle^2_g}}
   {|\nabla u|_g\sqrt{1-\langle N_{\pa E},\frac{\nabla u}
     {|\nabla u|_g}\rangle_g^2}} \,d\la_{\pa E^s}\,ds 
 \\
  &= \int_\R \int_{\Lambda^s} 
 \frac{ |\nablaH u|_g}{|\nabla u|_g}\,  \vartheta^s\, 
   d\la_{\pa E^s}\,ds,
\end{split}
\end{equation}
where we let 
\[
\vartheta^s  = \frac{\sqrt{|N_{\pa E}^H|_g^2-
   \langle N_{\pa E}^H,\frac{\nablaH u}{|\nablaH u|_g}\rangle_g^2}}
   {\sqrt{1-\langle N_{\pa E},\frac{\nabla u}{|\nabla u|_g}\rangle^2_g}} .
\]
We will prove in \emph{Step 3} that, for any $s\in\R$ 
such that $\nabla u\neq 0$ on $\Si^s$, there holds
\begin{equation}\label{eq:fattore}
\vartheta^s =
 |N_{\pa E^s}^{H\Si^s}|_g\quad\text{on }
\Lambda^s.
\end{equation}
Using  \eqref{eq:fattore}, \eqref{eq:pppperimetri}, and
\eqref{eq:a.e.s222}
formula \eqref{godot2} becomes
\begin{equation}
\label{godot3}
\begin{split}
\int_\R \int_{\partial E\cap \Si^s} h\,d\la_{\pa E\cap \Si^s}\,ds  
&= \int_\R \int_{\Lambda^s}
 \frac{ |\nablaH u|_g }{|\nabla u|_g}  |N_{\pa E^s} ^{H\Si^s}|_g\,
  d\la_{\pa E^s}\,ds 
 \\
& =   \int_\R \int_{\Lambda^s} 
 \frac{ |\nablaH u|_g }{|\nabla u|_g}  \,d\mu_{E}^{s}\,ds
 \\
&
=  \int_\R \int_{\partial E\cap \Si^s} 
 \frac{ |\nablaH u|_g }{|\nabla u|_g}  \,d\mu_{E}^{s}\,ds.
\end{split}
\end{equation}
The proof is complete,
 because \eqref{eq:coarealiscia} follows 
from \eqref{eq:coareaclass}, \eqref{eq:godot1} and \eqref{godot3}.

\medskip

{\em Step 3.} We prove claim  \eqref{eq:fattore}. 
Let us introduce the vector field $W$
in $\Omega \setminus \{ \nablaH u =0\}$
\[
   W=\frac{Tu}{|\nabla u|_g}\frac{\nablaH u}{|\nablaH u|_g}
    - \frac{|\nablaH u|_g}{|\nabla u|_g}T.
\]
It can be checked that $|W|_g=1$ and $Wu=0$.
In particular, for a.e.~$s$ we have $W\in T\Si^s$. 
Moreover, $W$ is $g$-orthogonal to $H\Si^s$ 
because any vector in $H\Si^s$ is orthogonal both to 
$\nablaH u$ and to $T$. It follows that
\[
N_{\pa E^s}^{H\Si^s} = N_{\pa E^s} - \langle N_{\pa E^s}, W\rangle_g
\]
and, in particular,
\[
|N_{\pa E^s}^{H\Si^s}|_g^2 = 1 - \langle N_{\pa E^s}, W\rangle_g^2.
\]

Starting from the formula 
\[
N_{\pa E^s} = 
\frac{N_{\pa E}-\langle N_{\pa E},\frac{\nabla u}{|\nabla
u|_g}\rangle_g\frac{\nabla u}{|\nabla u|_g}}{|N_{\pa E}-\langle N_{\pa
E},\frac{\nabla u}{|\nabla u|_g}\rangle_g\frac{\nabla u}{|\nabla u|_g}|_g} = 
\frac{N_{\pa E}-\langle N_{\pa E},\frac{\nabla u}{|\nabla
u|_g}\rangle_g\frac{\nabla u}{|\nabla u|_g}}{\sqrt{1-\langle N_{\pa
E},\frac{\nabla u}{|\nabla u|_g}\rangle_g^2}},
\]
we find 
\[
  |N_{\pa E^s}^{H\Si^s}|_g^2 = \frac{M}{1-\langle N_{\pa E},\frac{\nabla
u}{|\nabla u|_g}\rangle_g^2},
\]
where we let
\[
  M = 1-\langle N_{\pa E},\frac{\nabla u}{|\nabla u|_g}\rangle_g^2 
-\left\langle  N_{\pa E}-\langle N_{\pa E},\frac{\nabla u}{|\nabla u|_g}
\rangle_g\frac{\nabla u}{|\nabla u|_g}, W  \right\rangle_g^2.
\]
We claim that on the open set $\{\nablaH u\neq0\}$ there holds 
\begin{equation}\label{eq:casino}
\begin{split}
M=
 |N_{\pa E}^H|_g^2-\langle N_{\pa E}^H,\tfrac{\nablaH u}{|\nablaH u|_g}
 \rangle^2_g,
\end{split}
\end{equation}
and formula  \eqref{eq:fattore} follows from \eqref{eq:casino}.
Using the 
identity
$\nabla  u = \nablaH u + (Tu) T$
 and the 
orthogonality
\[
\left\langle N_{\pa E} - \langle N_{\pa E},\frac{\nabla u}{|\nabla u|_g}
\rangle_g\frac{\nabla u}{|\nabla u|_g},  \nabla u \right\rangle_g=0,
\]
we find 
\begin{equation}\label{eq:bici222}
\begin{split}
M &= 
 1-\left \langle N_{\pa E},\frac{\nablaH u + (Tu)T}
{|\nabla u|_g}\right \rangle_g^2 - 
\Big(\  \frac{Tu}{|\nabla u|_g} \langle N_{\pa E},\frac{\nablaH u}
{|\nablaH u|_g}\rangle_g
- \frac{|\nablaH u|_g}{|\nabla u|_g}\langle N_{\pa E},T\rangle_g 
\Big)^2
\\
& 
=1- \langle N_{\pa E},\frac{\nablaH u}{|\nablaH u|_g}\rangle_g^2 
\frac{|\nablaH u|_g^2+(Tu)^2}{|\nabla u|_g^2} - 
\langle N_{\pa E},T\rangle_g^2\frac{|\nablaH u|_g^2+(Tu)^2}
{|\nabla u|_g^2} 
\\
& =  1- \langle N_{\pa E},\frac{\nablaH u}{|\nablaH u|_g}\rangle_g^2  
- \langle N_{\pa E},T\rangle_g^2
\\
&
= 1- \langle N_{\pa E},T\rangle_g^2 - 
\Big( \langle N_{\pa E},\frac{\nablaH u}{|\nablaH u|_g}\rangle_g
- \Big\langle\langle N_{\pa E},T\rangle_g 
 { T,\frac{\nablaH u}{|\nablaH u|_g}\Big\rangle_g}  \Big)^2
\\
&=
|N_{\pa E}^H|_g^2-\langle N_{\pa E}^H,\frac{\nablaH u}{|\nablaH u|_g}
 \rangle^2_g.
\end{split}
\end{equation}

\noindent This ends the proof.
\end{proof}

We prove a coarea inequality.

\begin{proposition}\label{prop:disugcoarea}
Let $\Omega\subset\Hn$ be an open set,
 $u\in C^\infty(\Omega)$ a smooth function,
 $E\subset\Hn$  a set with finite $H$-perimeter in $\Omega$,
 and let $h:\pa E\to [0,\infty]$ be a Borel function. Then we 
have 
\begin{equation}
\label{ineq:coareagen222}
\int_\R \int_\Omega h\,\frac{|\nablaH u|_g}
 {|\nabla u|_g}\,d\mu_E^{s}\,ds 
   \leq
\int_{\Omega}h\,\sqrt{|\nablaH u|_g^2-\langle \nu_E,\nablaH u\rangle^2_g}\,
 d\mu_E.
\end{equation}
\end{proposition}

\begin{proof} The coarea inequality
\eqref{ineq:coareagen222} follows from the smooth case of Lemma
\ref{lem:coarealiscia} by an approximation and lower semicontinuity argument.

{\em Step 1. } By \cite[Theorem 2.2.2]{FSSChouston},
there exists a sequence of smooth
sets $(E_j)_{j\in\N}$  in $\Omega$ such that
\[
\chi_{E_j}\stackrel{L^1(\Omega)}{\to}\chi_E\quad\text{as }j\to\infty
\qquad\text{and}\qquad
\lim_{j\to\infty}\mu_{E_j}(\Omega)=\mu(\Omega).
\]
By a straightforward adaptation of the proof 
of \cite[Proposition 3.13]{AFP}, we also have
 that $\nu_{E_j} \mu_{E_j}\to \nu_E \mu_E$ weakly$^\ast$ in $\Omega$.
Namely, for any 
$\psi\in C_c(\Omega;H)$ there holds 
\[
\lim_{j\to\infty} \int_\Omega \langle \psi, \nu_{E_j}\rangle_g
 \,d\mu_{E_j} = \int_\Omega \langle \psi,\nu_E\rangle _g \,
d\mu_E.
\]

Let $A\subset\subset \Omega$ be an open set 
such that $\lim_{j\to\infty}\mu_{E_j}(A)= \mu_E (A)$. 
By  Reshetnyak's continuity theorem (see e.g.~\cite[Theorem 2.39]{AFP}),
we have 
\[
\lim_{j\to\infty}\int_{A}f(p,\nu_{E_j}(p))\,d\mu_{E_j} 
 = \int_{A}f(p,\nu_{E}(p))\,d \mu_{E} 
\]
for any continuous and bounded function $f$. In particular, 
\begin{equation}\label{eq:lagna}
 \lim_{j\to\infty} \int_{A}\,
 \sqrt{|\nablaH u|_g^2-\langle \nu_{E_j},\nablaH u\rangle^2_g}\,d\mu_{E_j}
   =   \int_{A}\,
    \sqrt{|\nablaH u|_g^2-\langle \nu_E,\nablaH u\rangle^2_g}\,d \mu_E .
\end{equation}

\medskip

{\em Step 2.} Let $(E_j)_{j\in\N}$ be the sequence introduced in {\em Step 1}.
Then, for a.e.~$s\in\R$ we have
\[
\text{$\nabla u\neq 0$ on $\Sigma^s$}\qquad\text{and}
 \qquad\chi_{E_j}{\to}\chi_E\text{ in $L^1(\Si^s,\si_{\Si^s})$ as }
 j\to\infty.
\]
In particular, for any such $s$ and for any open set
$A\subset \Si^s\cap\Omega$ 
there holds
\[
 \mu_E^{s} (A)\leq \liminf_{j\to\infty} \mu_{E_j}^{s}(A).
\]

From Fatou's Lemma and from 
the continuity of $\tfrac{|\nablaH u|_g}{|\nabla u|_g}$ on $\Si^s$, 
it follows that
\begin{align*}
\int_{A} \frac{|\nablaH u|_g}{|\nabla u|_g}\,d\mu_E^{s} 
& = \int_0^\infty \mu_E^{s}
 \left( \left\{p\in A: \tfrac{|\nablaH u|_g}{|\nabla u|_g}(p)> t\right\}
 \right)\,dt
 \\
& \leq \int_0^\infty \liminf_{j\to\infty}\mu_{E_j}^{s}
 \left( \left\{p\in A 
 :\tfrac{|\nablaH u|_g}{|\nabla u|_g}(p)> t\right\}\right)\,dt
  \\
& \leq \liminf_{j\to\infty}\int_0^\infty \mu_{E_j}^{s}
 \left( \left\{p\in A :\tfrac{|\nablaH u|_g}{|\nabla u|_g}(p)> t\right\}
 \right)\,dt\\
&= \liminf_{j\to\infty} \int_{A} \frac{|\nablaH u|_g}{|\nabla u|_g}\,
 d\mu_{E_j}^{s}.
\end{align*}
Using again Fatou's Lemma and Lemma \ref{lem:coarealiscia},
\begin{align*}
\int_\R \int_{A} \frac{|\nablaH u|_g}{|\nabla u|_g}
\,d\mu_E^{s}\,ds
&  \leq \int_\R \liminf_{j\to\infty} \int_{A} \frac{|\nablaH u|_g}
{|\nabla u|_g}\,d\mu_{E_j}^{s}
\,ds
\\
& \leq  \liminf_{j\to\infty} \int_\R \int_{A} 
\frac{|\nablaH u|_g}{|\nabla u|_g}\,d\mu_{E_j}^{s}
\,ds
\\
& = \liminf_{j\to\infty} \int_{A}\,
\sqrt{|\nablaH u|_g^2-\langle \nu_{E_j},\nablaH u\rangle^2_g}\,
d\mu_{E_j}.
\end{align*}
This, together with \eqref{eq:lagna}, gives
\[
\int_\R \int_{A} \frac{|\nablaH u|_g}{|\nabla u|_g}\,d\mu_E^{s}
\,ds 
\leq   \int_{A}\:\sqrt{|\nablaH u|_g^2-\langle \nu_E,\nablaH u
\rangle^2_g}\,d\mu_E.
\]

\medskip

{\em Step 3.} Any open set $A\subset\Omega$ 
can be approximated by a sequence $(A_k)_{k\in\N}$ 
of open sets such that
\[
A_k\subset\subset\Omega,\quad A_k\subset A_{k+1},\quad \bigcup_{k=1}^\infty 
A_k=A\quad\text{and}\quad \mu_E(\partial A_k)=0.
\]
In particular, for each $k\in\N$ we have  
\begin{align*}
\liminf_{j\to\infty} \mu_{E_j}(A_k) & \leq \limsup_{j\to\infty} 
 \mu_{E_j}(\overline{A}_k) \leq \mu_E(\overline{A}_k)
  \\
& = \mu_E(A_k) \leq \liminf_{j\to\infty} \mu_{E_j}(A_k).
\end{align*}
Hence, the inequalities are equalities, i.e., 
$\mu_E(A_k) = \displaystyle \lim_{j\to\infty} \mu_{E_j}(A_k)$. 
By \emph{Step 2},  for any $k\in\N$  there holds
\[
\int_\R \int_{A_k} \frac{|\nablaH u|_g}{|\nabla u|_g}\,d\mu_E^{s}\,ds 
\leq   \int_{A_k}\,
 \sqrt{|\nablaH u|_g^2-\langle \nu_E,\nablaH u\rangle^2_g}\,d\mu_E.
\]
By monotone convergence, letting  $k\to\infty$ we obtain for any open set
$A \subset\Omega$ 
\[
        \int_\R \int_A \frac{|\nablaH u|_g}{|\nabla u|_g}\,
   d \mu_E^{s}\,ds 
 \leq   \int_{A}\,
 \sqrt{|\nablaH u|_g^2-\langle \nu_E,\nablaH u\rangle^2_g}\,d\mu_E.
\]

By a standard approximation argument, it is
enough to prove \eqref{ineq:coareagen222} for the
characteristic function $h=\chi_B$ 
of a Borel set $B\subset\pa E$. Since the measure
$
\sqrt{|\nablaH u|_g^2-\langle \nu_{E},\nablaH u\rangle^2_g}\,\mu_E$
is a Radon measure on $\pa E$, there exists a sequence 
of open sets $B_j$ such that $B\subset B_{j}$ for any $j\in\N$ and
\[
    \lim_{j\to\infty} \int_{B_j}\,
      \sqrt{|\nablaH u|_g^2-\langle \nu_{E},\nablaH u\rangle^2_g}\,
 d \mu_{E} 
 = \int_{B}\,\sqrt{|\nablaH u|_g^2-\langle \nu_{E},\nablaH u\rangle^2_g}\,
 d \mu_{E}.
\]
Therefore, we have 
\begin{align*}
     \int_\R \int_B \frac{|\nablaH u|_g}{|\nabla u|_g}\,d \mu_E^{s}\,
         ds
& \leq  \liminf_{j\to\infty} \int_\R \int_{B_j} \frac{|\nablaH u|_g}
       {|\nabla u|_g}\,d \mu_{E}^{s}\,ds 
   \\
& \leq \lim_{j\to\infty} \int_{B_j}\,
    \sqrt{|\nablaH u|_g^2-\langle \nu_{E},\nablaH u\rangle^2_g}\,d \mu_{E}
    \\
& = \int_{B}\,\sqrt{|\nablaH u|_g^2-\langle \nu_{E},\nablaH u\rangle^2_g}\,
 d \mu_{E},
\end{align*}
and this concludes the proof.
\end{proof}

In the next step, we  prove an approximate  coarea formula for sets $E$ such that the boundary
$\partial E$
is an $H$-regular surface.

\begin{lemma}\label{lem:approxcoarea}
Let $\Omega\subset\Hn$ be an open set, $u\in C^\infty(\Omega)$
a smooth function,   $E\subset\Hn$  an open set
such that $\partial E\cap\Omega$ is an $H$-regular hypersurface,
and   $\bar p\in\partial E\cap\Omega$ a point  such that
\[
\nablaH u(\bar p)\neq 0\quad\text{and}\quad \nu_E(\bar p)
\neq \pm\frac{\nablaH u(\bar p)}{|\nablaH u(\bar p)|_g}.
\]
Then, for any $\ep>0$ there exists $\bar r=\bar r(\bar p,\ep)>0$ 
such that $B_{\bar r}(\bar p)\subset\Omega$ and, for any $r\in(0,\bar r)$, 
\begin{multline*}
\label{ACF}
  (1-\ep)\int_{B_r(\bar p)}
  \sqrt{|\nablaH u|^2_g-\langle \nu_E,\nablaH u\rangle_g^2}\,d\mu_E
\\
\leq \int_\R \int_{B_r(\bar p)}\frac{|\nablaH u|_g}{|\nabla u|_g}\,
  d\mu_E^{s}
        \,ds
  \\ 
\leq (1+\ep)\int_{B_r(\bar p)}\sqrt{|\nablaH u|^2_g-\langle \nu_E,
  \nablaH u\rangle_g^2}\,d\mu_E.
\end{multline*}
\end{lemma}

\begin{proof}
We can without loss of generality assume that $\bar p=0$ and $u(0)=0$. 
We divide the proof into several steps. 
\medskip

{\em Step 1: preliminary considerations.} 
The horizontal vector field $V_{2n}=\frac{\nablaH u}{|\nablaH u|_g}$ 
is well defined 
in a neighbourhood $\Omega_\ep\subset \Hn$ of 0.
For any $s\in\R$, 
the hypersurface $\Si^s = \{ p\in \Omega : u(p)=s\}$
is smooth in $\Omega_\ep$  
because $\nablaH u\neq 0$ on $\Omega_\ep$.

There are horizontal vector fields $V_1,\dots,V_{2n-1}$ on $\Omega_\ep$ such
that
$V_1,\dots,V_{2n}$ is a $g$-orthonormal frame. 
In particular, we have $V_j u=0$ for all $j=1,\dots,2n-1$, i.e., 
\begin{equation}
  \label{eq:subCC}
H_p\Si^{s} = \text{span}\{V_1(p),\dots,V_{2n-1}(p)\}
 \quad\text{for all }p\in \Sigma^s \cap \Omega_\ep.
\end{equation}

Possibly  shirinking $\Omega_\ep$, reordering $\{V_j\}_{j=1,\dots,2n-1}$, 
and changing the sign of $V_1$, we can assume 
(see \cite[Lemma 4.3 and Lemma 4.4]{V}) 
that there exist a function $f:\Omega_\ep\to\R$ and a number $\delta>0$ 
such that:
\begin{itemize}
\item[a)]
$f\in C^1_H(\Omega_\ep)\cap C^\infty(\Omega_\ep\setminus\partial E)$;

\item[b)]$E\cap\Omega_\ep = \{p\in\Omega_\ep: f(p) >0 \}$;
 
\item[c)]  $V_1 f\geq \delta >0$  on  $\Omega_\ep$.
\end{itemize}
\noindent
By \cite[Remark 4.7]{V}, we 
have also $\nu_E=\frac{\nablaH f}{|\nablaH f|_g}$ 
on $\partial E \cap \Omega_\epsilon$.

\medskip

{\em Step 2: change of coordinates.} Let $S\subset\Hn$
be a $(2n-1)$-dimensional smooth submanifold such that: 
\begin{itemize}
\item[i)] $0\in S$;
\item[ii)] $S\subset\Si^0\cap\Omega_\ep$; 
  in particular, $\nabla u$ is $g$-orthogonal to $S$;
\item[iii)] $V_1(0)$ is $g$-orthogonal to $S$ at $0$;
\item[iv)] 
there exists a diffeomorphism  
$H:U\to \Hn$, where
$U\subset \R^{2n-1}$ is an open set with  $0\in U$, such that
$H(0)=0$ and $H(U)= S\cap \Omega_\ep$; 
\item[v)] the area element  $JH$  of $H$ 
satisfies $JH(0)=1$. Namely, there holds
\[
 JH(0)=  \lim_{r\to0} 
 \frac{\lambda_S(H(B^E_r))}
{\mathcal L^{2n-1}(B^E_r)}=1,
\]
where $B^E_r = \{p\in\R^{2n-1}:|p|<r\}$ 
is a Euclidean ball   and
$\lambda_S$ is the Riemannian $(2n-1)$-volume measure on $S$ induced by $g$.  
\end{itemize}

\noindent 
For small enough $a,b>0$ and possibly shirinking
$U$ and $\Omega_\ep$, the mapping
$G : (-a,a)\times(-b,b)\times U\to \Hn$
\[ 
 G(v,z,w) = \exp(vV_1)\exp\big(z\tfrac{\nabla u}{|\nabla u|^2_g}\big)(H(w)) 
\]
is a diffeomorphism from  
$\wt\Omega_\ep=(-a,a)\times(-b,b)\times U$ onto $\Omega_\ep$. 
The differential of $G$ satisfies  
\[
  dG\Big(\frac{\partial}{\partial v} \Big)=V_1\quad \text{and}
   \quad 
 dG(0)\Big(\frac{\partial}{\partial  z} \Big)
 =\frac{\nabla u(0)}{|\nabla u(0)|_g^2}.
\]
Moreover, the tangent space $T_0S   =    \text{Im}\,dH(0)$
  is $g$-orthogonal to $V_1(0)$ and 
$\tfrac{\nabla u(0)}{|\nabla u(0)|_g^2}$.
We denote by $G_z$ the restriction of $G$
to
$(-a,a)\times\{z\}\times U$, i.e., $G_z(v,w) = G(v,z,w)$.
From the above considerations, we
deduce that the area elements of $G$ and of $G_0$
satisfy 
\[
JG(0)=\frac{1}{|\nabla u(0)|_g} \quad\text{and}\quad JG_0(0)=1.
\]
Then, possibly shrinking further $\wt\Omega_\ep$, we have 
\begin{align}
(1-\ep)JG(v,z,w)\leq\frac{JG_{z}(v,w)}
 {\big|\nabla u\circ G(v,z,w)\big|_g}\leq (1+\ep)JG(v,z,w),
\label{pennetta0}
\end{align}
for all $(v,z,w)\in\wt \Omega_\ep$.

For $j=1,\dots,2n$, we
define on $\wt \Omega_\ep$
the vector fields $\wt V_j= (d G)^{-1}(V_j)$.
By the definition of $G$, we have 
$\wt V_1={\partial} / {\partial v}$. 
We also define $\wt u=u\circ G\in C^\infty(\wt\Omega_\ep)$,
$\wt f=f\circ G:\wt\Omega_\ep\to\R$, and 
$\wt E=G^{-1}(E)$.
Then:
\begin{itemize}

\item[1)] we have $\wt E= \{ q\in\wt\Omega_\ep:  \wt f(q) >0\}$;

\item[2)] we have $\wt f\in C^\infty(\wt\Omega_\ep\setminus\partial \wt E)$;

\item[3)] the derivative  $\wt V_j\wt f$ is defined in the sense of
distributions
with respect to the measure $\mu = JG \mathcal L^{2n+1}$. 
Namely,  for all $\psi \in C^\infty_c(\wt \Omega_\ep)$
we have
\[
 \int_{\wt \Omega_\ep} (\wt V_j\wt f) \, \psi \, d\mu = - 
 \int_{\wt\Omega_\ep}  \wt f \,  \wt V_j^* \psi \, d\mu, 
\]
where $\wt V_j^*$ is the adjoint operator of $\wt V_j$ with respect to $\mu$.
Then  we have $\wt V_j\wt f=(V_jf)\circ G$  
and so $\wt V_j\wt f$ is a continuous function
for any $j=1,\dots,2n$.
In particular, we have 
$\wt V_1 \wt f=\partial_v\wt f\geq \delta >0$.\end{itemize}

\medskip

{\em Step 3: approximate coarea formula.}
We follow the argument of
\cite[Propositions 4.1 and 4.5]{V},  see also Remark 4.7 therein.

Possibly shrinking $\wt \Omega_\ep$ and $\Omega_\ep$, 
there exists a continuous function $\phi:(-b,b)\times U\to (-a,a)$ 
such that:
\begin{itemize}

\item[A)] $\partial\wt E\cap \wt\Omega_\ep$ 
is the graph of $\phi$. Namely, letting $\Phi: (-b,b)\times U\to\R^{2n+1}$, 
$\Phi(z,w)=(\phi(z,w),z,w)$, we have:
\[
      \partial \wt E\cap\wt\Omega_\ep 
     = \Phi((-b,b)\times U).
\]

\item[B)]
The measure $\mu_E$ is
\begin{equation}
 \label{perE}
     \mu_E\res\Omega_\ep 
   = (G\circ \Phi)_\# 
      \left( \bigg( \frac{|\wt V\wt f|}
  {\wt V_1\tilde f}JG\bigg)\circ \Phi\ \mathcal L^{2n}\res((-b,b)\times
U)\right),
\end{equation}
where $(G\circ\Phi)_\#$ denotes the push-forward and
\[
|\wt V\tilde f| = \Big( \sum_{j=1}^ {2n} \big(\wt V_j \wt f\big)^2 \Big)^{1/2}.
\]
\end{itemize}

Using   $V_1u=0$ and $u\circ H=0$ (this follows from
$H(U)=S\cap\Omega_\ep\subset
\Sigma^0\cap \Omega_\ep$), we obtain 
\[
\begin{split}
\wt u(v,z,w) &= u(G(v,z,w)) 
      =  u\big(\exp(vV_1)\exp\big(z\tfrac{\nabla u}{|\nabla
u|^2_g}\big)(H(w))\big) 
\\
&
= u\big(\exp\big(z\tfrac{\nabla u}{|\nabla u|^2_g}\big)(H(w))\big)
\\
& = z+u(H(w))  = z.
\end{split}
\]
In particular, from $\wt u = u\circ G$ we deduce that
\[
G^{-1}(\Sigma^s\cap\Omega_\ep)= (-a,a)\times\{s\}\times U.
\]
 
We denote by $J G_s$ the Jacobian (area element) of $G_s$.
We also define the restriction $\Phi_s:U\to\R^{2n+1}$, $\Phi_s(w)=\Phi(s,w)$,
for any $s\in (-b,b)$.

By \eqref{eq:subCC}, for any $s\in\R$
the measure $\mu_E^s =  \mu_{E\cap \Sigma_s}^{\Sigma^s}$ 
is the horizontal
perimeter of $E\cap \Si^s$ with respect to the Carnot-Carath\'eodory 
structure
induced by the family $V_1,\dots,V_{2n-1}$ on $\Si^s$.
We can repeat the argument that lead to \eqref{perE} to obtain 
\begin{equation}\label{perEs}
\mu_E^{s}\res\Omega_\ep = (G\circ \Phi_s)_\# 
  \left( \bigg( \frac{|\wt
V'\wt f|}{\wt V_1\wt f}\,JG_s \bigg)\circ
\Phi_s  
 \mathcal L^{2n-1}\res U \right),
\end{equation}
where $\wt V'\wt f=(\wt V_1\wt f,\dots,\wt V_{2n-1}\wt f)$. 
We omit details of the proof of \eqref{perEs}.
The proof is a line-by-line repetition of Proposition 4.5 in
\cite{V} with the unique difference that now the horizontal perimeter
is defined in a curved manifold.

Let us fix $\bar r>0$ such that $B_{\bar r}\subset\Omega_\ep$,
and for any  
$r\in(0,\bar r)$ let   
\[
\begin{split}
& A_{s,r}=\big\{ w\in U :  G(0,s,w) \in B_r\big\},
\\
&   A_r  = \big\{ (s,w) \in (-b,b)\times U : w\in A_{s,r}\big\}.
\end{split}
\]
By Fubini-Tonelli theorem and by 
\eqref{perEs}, the function
\begin{equation}\label{MIS}
   s\mapsto \int_{B_r}
  \frac{|\nablaH u|_g}{|\nabla u|_g}\,d\mu_{E}^{ s}
=
\int_{A_{s,r}}
 \bigg(\frac{|\nablaH u|_g}{|\nabla u|_g}
\circ G 
\bigg)\bigg( \frac{|\wt V'\wt f|}{\wt V_1
\wt f}\,JG_s  \bigg)\circ\Phi_s\, d\mathcal L^{2n-1} 
\end{equation}
is $\mathcal L^1$-measurable. Here and hereafter, 
the composition  $\circ\Phi_s$ acts on the product.
Thus, from  Fubini-Tonelli theorem and \eqref{pennetta0} we obtain 
\begin{equation}\label{SNS22}
\begin{split}
  \int_\R \int_{B_r}
  &     \frac{|\nablaH u|_g}{|\nabla u|_g}\,d\mu_{E}^{ s}\,ds 
   =  
       \int_\R \int_{A_{s,r}}
 \bigg(\frac{|\nablaH u|_g}{|\nabla u|_g}
\circ G 
\bigg)\bigg( \frac{|\wt V'\wt f|}{\wt V_1
\wt f}\,JG_s  \bigg)\circ\Phi_s(w)\, d\mathcal L^{2n-1}(w) 
\,ds
\\
&= 
\int_{A_{r}}
 ( |\nablaH u|_g 
\circ G) 
  \bigg( \frac{|\wt V'\wt f|}{\wt V_1
\wt f}\,
\frac{JG_s }{|\nabla u|_g\circ G}
 \bigg)\circ\Phi(s,w)\, d\mathcal L^{2n}(s,w)
\\
 & \leq  (1+\ep)
 \int_{A_r }
(|\nablaH
u|_g\circ G) \bigg( \frac{|\wt V\wt f|}{\wt V_1\wt f}
\sqrt{1-\tfrac{(\wt V_{2n}\wt f)^2}{|\wt V\wt f|^2}}\, JG\bigg)\circ
\Phi(s,w)\,d\mathcal L^{2n}(s,w).
\end{split}
\end{equation}
From the identity
\begin{equation}
\label{SNS2}
\begin{split}
\frac{\wt V_{2n}\wt f}{|\wt V\wt f|} & 
 = \frac{V_{2n}f}{|\nablaH f|_g}\circ G 
= \Big\langle \frac{\nablaH u}{|\nablaH u|_g}, \frac{\nablaH
f}{|\nablaH f|_g}\Big\rangle_g\circ G
= \Big\langle \frac{\nablaH u}{|\nablaH u|_g}, \nu_E\Big\rangle_g\circ G,
\end{split}
\end{equation}
and from 
\eqref{perE} we deduce that 
\begin{equation}\label{SNS23}
\begin{split}
  \int_\R \int_{B_r}
       \frac{|\nablaH u|_g}{|\nabla u|_g}\,d\mu_{E}^{ s}\,ds 
     & \leq   (1+\ep) \int_{B_r} |\nablaH u|_g\sqrt{1-\langle \tfrac{\nablaH
u}{|\nablaH u|_g}, \nu_E\rangle_g^2}\,d\mu_E
\\
& 
= (1+\ep)\int_{B_r}\sqrt{|\nablaH u|^2_g-\langle \nu_E,\nablaH
u\rangle_g^2}\,d\mu_E.
\end{split}
\end{equation}

In a similar way, we obtain
\[
\int_\R \int_{B_r}\frac{|\nablaH u|_g}{|\nabla u|_g}\,d\mu_{
E}^{s}\,ds 
\geq  (1-\ep)\int_{B_r}\sqrt{|\nablaH u|^2_g-\langle \nu_E,\nablaH
u\rangle_g^2}\,d\mu_E.
\]
This   concludes the proof.
\end{proof}

We can now prove the coarea formula for $H$-regular boundaries.

\begin{proposition}\label{prop:normalidiverse}
Let $\Omega\subset\Hn$ be an open set, $u\in C^\infty(\Omega)$, and
$E\subset\Hn$ be an
open domain such that $\partial E\cap\Omega$ is an $H$-regular hypersurface.
Then
\begin{equation}\label{muchocalor}
\int_\R \int_{\Omega}\frac{|\nablaH u|_g}{|\nabla u|_g}\,d\mu_E^{s}\,ds
= \int_{\Omega}\sqrt{|\nablaH u|_g^2-\langle \nu_E,\nablaH
u\rangle^2_g}\,d\mu_E.
\end{equation}
\end{proposition}
\begin{proof}
Let us define the set 
\[
A=\Big\{p\in\pa E\cap \Omega :\nablaH u(p)\neq 0\text{ and }\nu_E(p)\neq
\pm\frac{\nablaH u(p)}{|\nablaH u(p)|_g}\Big\}.
\]
The set $A$ is relatively open in $\pa E\cap \Omega$. Let $\ep>0$ be fixed.
Since the measure  $\mu_E$ is
locally doubling on $\partial E\cap \Omega$ (see e.g.~\cite[Corollary 4.13]{V}),
by Lemma \ref{lem:approxcoarea} and Vitali covering Theorem (see
e.g.~\cite[Theorem 1.6]{heinonen}) there exists a
countable (or finite) collection
 of balls $B_{r_i}(p_i)$, $i\in\N$, such that:
\begin{itemize}
\item[i)] for any $i\in\N$ we have $p_i\in A $ and $0<r_i<\bar r(p_i,\ep)$,
where $\bar r$ is as in
the statement of Lemma \ref{lem:approxcoarea};
\item[ii)] the balls $B_{r_i}(p_i)$ are contained in $A$ and pairwise disjoint;
\item[iii)] $\mu_E\big(A\setminus\bigcup_{i\in\N} B_{r_i}(p_i)\big)=0$.
\end{itemize}
It follows that we have:
\begin{equation}
\label{eq:epp1}
\begin{split}
   \int_\R\int_{\bigcup_{i\in\N}  B_{r_i}(p_i) }\frac{|\nablaH u|_g}{|\nabla
u|_g}\,d\mu_E^{s}\,ds 
& \leq  (1+\ep)\int_{\bigcup_{i\in\N} B_{r_i}(p_i)} \sqrt{|\nablaH
u|_g^2-\langle
\nu_E,\nablaH u\rangle_g^2}\,d\mu_E
\\
&= (1+\ep)\int_{A} \sqrt{\nablaH u|^2_g-\langle \nu_E,\nablaH
u\rangle^2_g}\,d\mu_E
\\
&= (1+\ep)\int_{\Omega} \sqrt{|\nablaH u|^2_g-\langle \nu_E,\nablaH
u\rangle^2_g}\,d\mu_E.
\end{split}
\end{equation}
The last equality follows from the fact that $\sqrt{|\nablaH u|_g^2-\langle
\nu_E,\nablaH u\rangle_g^2}=0$ outside $A$. In the same way one also obtains
\begin{equation}
\label{eq:epp2}
\int_\R\int_{\bigcup_{i\in\N}  B_{r_i}(p_i)}\frac{|\nablaH u|_g}{|\nabla
u|_g}\,d\mu_E^{s}\,ds 
\geq (1-\ep)\int_{\Omega} \sqrt{|\nablaH u|_g^2-\langle \nu_E,\nablaH
u\rangle_g^2}\,d\mu_E.
\end{equation}
Moreover, by Proposition \ref{prop:disugcoarea}, there holds
\[
\int_\R \int_{\Omega\setminus\bigcup_{i\in\N}  B_{r_i}(p_i)}
\frac{|\nablaH u|_g}{|\nabla u|_g}\,d\mu_E^{s}\,ds
\leq \int_{\Omega\setminus\bigcup_{i\in\N} B_{r_i}(p_i)}
 \sqrt{|\nablaH u|^2_g-\langle
\nu_E,\nablaH u\rangle_g^2}\,d\mu_E =0.
\]
In particular, the integral on the left hand side of the last inequality is 0
and, by \eqref{eq:epp1} and \eqref{eq:epp2}, we obtain
\begin{multline*}
(1-\ep)\int_{\Omega}\sqrt{|\nablaH u|^2_g-\langle \nu_E,\nablaH
u\rangle_g^2}\,d\mu_E\\
\leq \int_\R \int_{\Omega}\frac{|\nablaH u|_g}{|\nabla
u|_g}\,d\mu_E^{s}\,ds\\ 
\leq (1+\ep)\int_{\Omega}\sqrt{|\nablaH u|^2_g-\langle \nu_E,\nablaH
u\rangle_g^2}\,d\mu_E.
\end{multline*}
Since   $\ep>0$ is arbitrary, this  concludes the proof.
\end{proof}

By a standard approximation argument, we also have the following extension
of the coarea formula \eqref{muchocalor}.

\begin{proposition}\label{prop:maldigola}
Let $\Omega\subset\Hn$ be an open set, $u\in C^\infty(\Omega)$, and $E$ be an
open domain such that $\partial E\cap\Omega$ is an $H$-regular hypersurface.
Then, for any Borel function $h:\pa E\to[0,\infty)$ there holds
\[
\int_\R \int_{\Omega}h\, \frac{|\nablaH u|_g}{|\nabla
u|_g}\,d\mu_E^{s}\,ds
= \int_{\Omega}h\,
  \sqrt{|\nablaH u|^2_g-\langle \nu_E,\nablaH
u\rangle_g^2}\,d\mu_E.
\]
\end{proposition}

Our next step is to prove the coarea formula for $\mathcal S^{2n+1}$-rectifiable
sets. 

\begin{lemma}\label{lem:rettif-Borel}
Let $R\subset\Hn$ be an  $\mathcal S^{2n+1}$-rectifiable
set.   Then,  
there exists a Borel  $\mathcal S^{2n+1}$-rectifiable  set $R'\subset\Hn$ such
that $\SQmu(R\Delta R')=0$.
\end{lemma}

\begin{proof}
By assumption, there exist  a $\SQmu$-negligible set $N$ and $H$-regular
hypersurfaces $S_j\subset\Hn$, $j\in\N$,  such that
\[
   R\subset N\cup\bigcup_{j=1}^\infty S_j.
\]
It is proved in \cite{FSSCMathAnn,ASCV} that (up to a localization argument),
for any $j\in\N$, there exist  an open set $U_j\subset\R^{2n}$,  
an omeomorphism $\Phi_j:U_j\to S_j$, and 
a continuous function $\rho_j:U_j\to[1,\infty)$
such that $\SQmu\res S_j=\Phi_{j\#}(\rho_j\:\leb^{2n}\res U_j)$. 
Since the
Lebesgue measure $\leb^{2n}$ is a complete  Borel measure,
for any $j\in
\N$ there exists a Borel set $T_j\subset U_j$ such that
\[
\leb^{2n}(T_j \Delta \Phi_j^{-1}(R\cap S_j))=0.
\]
In particular, the Borel set
\[
R'=\bigcup_{j=1}^\infty \Phi_j(T_j)
\]
is $\SQmu$-equivalent to $R$.
\end{proof}

 \medskip

\begin{proof}[Proof of Theorem \ref{teo:rettificabili22}]
{\em Step 1.}  
We prove \eqref{muchofrio22} when  
$R$ is an $H$-regular hypersurface.
Then, $R$ is
locally the boundary of an open set 
$E\subset\Hn$ with $H$-regular boundary. Moreover, we have
(locally) $\mu_E=\SQmu\res R$ and $\nu_E=\nu_R$, up to the sign.  

We define the measures   $\mu_R^{s}=\mu_E^{s}$ for any $s$ such that $\nabla u\neq 0$ on $\Si^s$.
The measurability of the function in \eqref{filippo}
follows 
from the argument \eqref{MIS}. Formula \eqref{muchofrio22}
follows 
from  Proposition
\ref{prop:maldigola}.

\medskip

{\em Step 2.}  
We prove \eqref{muchofrio22} when  $R$ is an $\mathcal S^{2n+1}$-rectifiable Borel set. There exist
a
$\SQmu$-negligible set $N$ and $H$-regular hypersurfaces $S_j\subset\Hn$, 
$j\in\N$ such that
\[
                 R\subset N\cup\bigcup_{j=1}^\infty S_j.
\]
Each $S_j$ is (locally) the boundary of an open set $E_j$ 
with $H$-regular boundary. We denote by $\mu_{E_j}^s$ the perimeter
measure on $\partial E_j\cap \Sigma^s$ induced by $E_j$.

We define the pairwise disjoint Borel sets $R_j=(R\cap
S_j)\setminus\cup_{h=1}^{j-1}S_h$ and we let
\[
\mu_R^{s}=\sum_{j=1}^\infty \mu_{E_j}^{s}\res R_j.
\]
The definition is well posed for any $s$ such that $\nabla u\neq 0$ on
$\Si^s$. We have $\nu_R=\pm \nu_{E_j}$ $\SQmu$-a.e.~on 
$R_j$ and the sign of $\nu_R$ does not affect formula \eqref{muchofrio22}.
From the {\em Step 1}, for each $j\in\N$ the function
\[
  s\mapsto \int_{R_j} 
h\,\frac{|\nablaH u|_g}
 {|\nabla u|_g}\,d\mu_{E_j}^{s}
\]
is $\mathcal L^1$-measurable; here, we were allowed to utilize {\em Step 1} because $\chi_{R_j}$ is Borel regular. Thus also the function
\[
s\mapsto \int_{\Omega} 
h\,\frac{|\nablaH u|_g}
 {|\nabla u|_g}\,d\mu_{R}^{s} = \sum_{j=1}^\infty \int_{R_j} 
h\,\frac{|\nablaH u|_g}
 {|\nabla u|_g}\,d\mu_{E_j}^{s}
\]
is measurable.
  Moreover, we have
\[
\begin{split}
\int_\R \int_{\Omega}h\,\frac{|\nablaH u|_g}
 {|\nabla u|_g}\,d\mu_R^{s}\, ds
&   =  
      \sum_{j=1}^\infty
     \int_\R \int_{R_j} 
h\,\frac{|\nablaH u|_g}
 {|\nabla u|_g}\,d\mu_{E_j}^{s}\, ds
\\
&
=\sum_{j=1}^\infty
\int_{R_j}h\,\sqrt{|\nablaH u|_g^2-\langle \nu_R,\nablaH
u\rangle_g^2}\,d\SQmu
\\
&=
\int_{R}h\,\sqrt{|\nablaH u|_g^2-\langle \nu_R,\nablaH
u\rangle_g^2}\,d\SQmu.
\end{split}
\]

{\em Step 3.} Finally, if $R$ is $\mathcal S^{2n+1}$-rectifiable but not Borel, we set
$\mu_R^{s}=\mu_{R'}^{s}$, where $R'$ is a Borel set as in Lemma
\ref{lem:rettif-Borel}. 
Again, this definition is well posed for a.e.~$s\in\R$.
This concludes the proof.
\end{proof}

\subsection{Proof of Theorem \ref{teo:coarealiscia2}}
In this subsection we assume $n\geq 2$.

\begin{lemma}
   \label{lem:trascurabile}
For  $n\geq 2$, let $\Omega\subset\Hn$ be an open set, 
$u\in C^\infty(\Omega)$ a smooth function, 
$R\subset\Omega$ an $\mathcal S^{2n+1}$-rectifiable set. 
Then 
\[
      \SQmu(\{p\in R:\nablaH u(p)=0\text{ and }\nabla u(p)\neq 0\})=0.
\]
\end{lemma}

\begin{proof}
It is enough to prove the lemma when $R$ is an
$H$-regular hypersurface. Let
\[
   A=\big\{p\in R:\nablaH u(p)=0\text{ and }\nabla u(p)\neq 0\big\}.
\]
We claim that $\SQmu(A)=0$.

Let $p\in A$ be a fixed point 
and let $\nu_R(p)$ be the horizontal normal to $R$ at $p$.
Since $n\geq 2$, we have 
\[
\dim \{V(p)\in H_p:\langle V(p),\nu_R(p)\rangle_g=0\}=2n-1\geq n+1.
\]
Thus there exist left invariant horizontal vector fields
$V,W$ such that
\[
    \langle V(p),\nu_R(p)\rangle_g=\langle
W(p),\nu_R(p)\rangle_g=0\quad\text{and}\quad [V,W]=T.
\]

From $\nablaH u(p)=0$ and $\nabla u(p)\neq 0$ we deduce that $Tu(p)\neq 0$.
It follows that  
\[
      VWu(p)-WVu(p)=Tu(p)\neq 0,
\]
and, in particular, we have either $VWu(p)\neq0$ or $WVu(p)\neq0$.
Without loss of
generality, we assume that $VWu(p)\neq0$. Then the
set
 $S=\{q\in\Omega:Wu(q)=0\}$
is an $H$-regular hypersurface near the point $p\in S$.
Since we have
\[
   \langle V(p), \nu_R(p)\rangle_g =0\quad\text{and}
\quad
   \langle V(p),\nu_S(p)\rangle_g = \frac{VWu(p)}{|\nablaH Wu(p)|_g} \neq 0,
\]
we deduce that  $\nu_R(p)$ and $\nu_S(p)$ are linearly independent.
Then there exists $r>0$ such that 
the set 
$R\cap S\cap B_r(p)$ is a 2-codimensional $H$-regular surface (see
\cite{FSSC-adv}). Therefore, by \cite[Corollary 4.4]{FSSC-adv}, 
the Hausdorff
dimension in the Carnot-Carath\'eodory metric 
of $A\cap B_r(p)\subset R\cap S\cap B_r(p)$ is not greater than $2n$.
This is enough to conclude.
\end{proof}

\begin{remark}\label{rem:notrascH1}
Lemma \ref{lem:trascurabile} is not valid in the case $n=1$.
Consider the smooth surface $R=\{(x,y,t)\in \H^1:x=0\}$ and the
function $u(x,y,t)=t-2xy$. We have 
\[
\nabla u = -4x Y +T\quad\textrm{and}\quad
 \nablaH u = -4x Y.
\]
Then we have
\[
\big\{p\in R:\nablaH u(p)=0\text{ and }\nabla u(p)\neq 0\big\} = R
\]
and $\mathcal S^3 (R) = \infty$.

\end{remark}

If  $n\geq 2$ and
$\Omega$, $u$, and $R$ are as in Lemma \ref{lem:trascurabile}, 
then the function 
\[
|\nabla u|_g\sqrt{1-\left\langle \nu_E,\tfrac{\nablaH u}{|\nablaH
u|_g}\right\rangle_g^2}
\]
is defined $\SQmu$-almost everywhere on $R$.
We agree that  its value is 0
when $|\nabla u|_g=0$.  Notice that, in this case, $\tfrac{\nablaH
u}{|\nablaH u|_g}$ is not defined.

\begin{proof}[Proof of Theorem \ref{teo:coarealiscia2}]
Let $\ep>0$ be fixed. Then \eqref{eq:coarealiscia2a} can be 
obtained by plugging the function  $\tfrac{|\nabla u|_g}
{\ep+|\nablaH u|_g}\,h$ into
\eqref{muchofrio22}, letting $\ep\to0$ and using the monotone
convergence theorem.
\end{proof}

\section{Height estimate} \label{tre}

In this section, we prove Theorem \ref{teoaltezza}. 
We discuss first a relative isoperimetric inequality on slices.
Then we list some elementary properties of 
excess, and finally we proceed with
the proof.

We assume throughout this section that $n\geq2$.

\subsection{Relative isoperimetric inequalities}
\label{trebis}

For each $s\in\R$, we define the level sets of the height function
\[
        \H^n_s=  \big\{p \in\Hn : \h(p)=s\big\}.
\]
Let $H^s$ be the $g$-orthogonal projection
of $H$ onto the tangent space of $\H^n_s$.
Since 
the vector field $X_1$ is orthogonal to $\H^n_s$, while the vector fields
$X_2,\ldots,X_n,Y_1,\ldots,Y_n$ are tangent to $\H^n_s$, then
at any point $p\in\H^n_s$
we have
\[
H_p^s =\mathrm{span} 
 \big\{ X_2(p) ,\ldots,X_n(p),Y_1^s(p),Y_2(p),\ldots,Y_n(p)\big\},
\]
where $X_2,Y_2,\ldots,X_n,Y_n$ are as in \eqref{XY} and
\[
       Y_1^s = \frac{\partial}{\partial y_1}-2s \frac{\partial}{\partial t}.
\]
The natural volume in $\H^n_s$ is the Lebesgue measure $\leb^{2n}$.
%
For any measurable set $F\subset \H^n_s$ and 
any open set $\Omega\subset\H^n_s$,
we define
\[
       \mu_F^s(\Omega) = \sup \Big\{ \int_F \mathrm{div}_{g}^s \varphi \, 
d\mathcal L^{2n}:
 \varphi \in C^1_c(\Omega;H^s),\, \|\varphi\|_g\leq 1\Big\},
\]
where $\mathrm{div}_{g}^s \varphi = X_2\varphi_2+\ldots+X_n\varphi_n
+Y_1^s\varphi_{n+1} +\ldots +Y_n\varphi_{2n}$.
If $\mu_F^s(\Omega)<\infty$ then $\mu_F^s$ is a Radon measure in $\Omega$.

By Theorem \ref{teo:coarealiscia2}, 
for any Borel function $h:\Hn\to[0,\infty)$ and any set $E$
with locally finite $H$-perimeter in $\Hn$, we have the following coarea
formula
\begin{equation}\label{eqcoarea}
\int_\R \int_{\mathbb H^n_{s}} h\,d \mu_{E^s}^s\, ds =  
\int_{\mathbb H^n }
   h\sqrt{1-\langle \nu_E,X_1\rangle_g^2  } \,d\mu_{E },
\end{equation}
where $E^s  = E\cap \H^n_s$ is the section of $E$
with $\H^n_s$.
Notice that $\nablaH \h = X_1$.

In the proof  of   Theorem \ref{teoaltezza}, we  
need a relative isoperimetric
inequality in each slice $\mathbb H^n_s$ for $s\in (-1,1)$.
These slices are cosets of $\W = \mathbb H^n_0$
and the isoperimetric inequalities in $\H^n_s$ can be reduced to an
isoperimetric inequality in the central slice $\W=\H^n_0$
relative to a family of varying domains.

For any $s\in (-1,1)$ let $\Omega_s\subset \W$ be the set 
$\Omega_s = (-s\ee_1)\ast D_1\ast (s\ee_1)$. 
This is the left translation by $-s\ee_1$
of the section  $C_1\cap \H^n_s$. See the introduction for the definition of $D_1$ and $C_1$.
With the coordinates $(y_1,\wh
z,t) \in \W = \R\times \C^{n-1}\times\R $, we have
\[
  \Omega_s = \big\{ (y_1,\wh z,t)\in \W : (y_1^2+|\wh
z|^2)^2+(t-4sy_1)^2<1\big\}.
\]
The sets $\Omega_s\subset\W$ are open and convex in the standard sense.
The boundary $\partial \Omega_s$ is a $(2n-1)$-dimensional
$C^\infty$ embedded surface with the following property.
There are $4n$ open convex sets $U_1,\ldots,U_{4n}\subset \W$ such that
$\partial \Omega_s\subset \bigcup _{i=1}^{4n} U_i$ and for each
$i$ the portion of boundary $\partial \Omega_s\cap U_i$ is a graph
of the form $p_j = f_i^s(\widehat p_j)$ with $j=2,\ldots, 2n+1$
and $\widehat p_j = (p_2,\ldots,p_{j-1},p_{j+1},\ldots,p_{2n+1}) \in V_i$,
where $V_i\subset\R^{2n-1}$ is an open convex 
set and $f_i^s\in C^\infty(V_i)$ is a function
such that 
\begin{equation}
\label{kappa}
 |\nabla f_i^s(\widehat p_j) - \nabla f_i^s(\widehat q_j)| \leq K
    |\widehat p_j-\widehat q_j|   \quad \textrm{for all }
\widehat p_j,\widehat q_j\in V_i,
\end{equation}
where $K>0$ is a constant independent of $i=1,\ldots,4n$ and
independent of $s\in (-1,1)$.
In other words, the boundary $\partial \Omega_s$ is of class $C^{1,1}$
uniformly in $s\in (-1,1)$.

By Theorem 3.2 in \cite{MM}, 
the domain $\Omega_s\subset\W$ is a non-tangentially accessible
(NTA) domain in the metric space $(\W,d_{CC})$ where $d_{CC}$ is
the Carnot-Carath\'eodory metric induced by the horizontal distribution $H^0_p$.
In particular, $\Omega_s$ is a (weak) John domain in the sense of  \cite{HK}.
Namely, there exist a point 
$p_0\in \Omega_s$, e.g.~$p_0=0$, 
and a constant $C_J>0$ such that for any point $p\in\Omega_s$ there exists
a continuous curve $\gamma:[0,1]\to \Omega_s$ 
such that $\gamma(1)=p_0$, $\gamma(0)=p$, and
\begin{equation}
     \mathrm{dist}_{CC}(\gamma(\sigma),
\partial\Omega_s)\geq C_J d_{CC}(\gamma(\sigma),p),\quad
\sigma\in[0,1].
\end{equation}
By Theorem 3.2 in \cite{MM}, the John constant $C_J$
depends only on the constant $K>0$ in \eqref{kappa}. 
This claim is not stated explicitly in Theorem 3.2 of \cite{MM}
but it is evident from the proof. 
In particular, the John constant $C_J$ is independent of $s\in (-1,1)$.
Then, by 
Theorem 1.22 in \cite{GN} we have the following result.

\begin{theorem} 
\label{pippo}
Let $n\geq 2$. There exists a constant $C(n)>0$
such that for any $s\in (-1,1)$ and for any 
 measurable set $F\subset \W$ there holds
\begin{equation}
 \label{ISOP1}
 \min \{   \leb^{2n}(F\cap \Omega_s), 
\leb^{2n}(\Omega_s\setminus F)\big\} ^
{\frac{2n}{2n+1}} \leq C(n)
\, \frac{ \mathrm{diam}_{CC}(\Omega_s) }{
\leb^{2n}(\Omega_s)^{\frac{1}{2n+1}}} \mu_F^0(\Omega_s).
\end{equation}
\end{theorem}

\noindent An alternative proof
of Theorem \ref{pippo} can be obtained 
using the Sobolev-Poincar\'e inequalities
proved in  \cite{HK}  in the general setting
of metric spaces.

The diameter 
$\mathrm{diam}_{CC}(\Omega_s)$ is bounded for $s\in(-1,1)$
and  $\leb^{2n}(\Omega_s)>0$ is a constant independent of $s$.
Then we obtain the following version of \eqref{ISOP1}.

\begin{corollary}\label{teoisoprelW} Let $n\geq 2$.
For any $\tau \in(0,1)$  there exists a constant 
$C (n,\tau)>0$ 
such that for $s\in(-1,1)$ and for
any measurable set $F\subset \W$ satisfying 
\[
\leb^{2n}(F\cap \Omega_s)\leq \tau  \,\leb^{2n}(\Omega_s) 
\]
there holds
\[
 \mu_{F}^0(\Omega_s) \geq C (n,\tau ) 
   \leb^{2n}(F\cap \Omega_s)^{\tfrac{2n}{2n+1}}.
\]
\end{corollary}

\subsection{Elementary properties of excess}
We list here, without proof, the most basic properties of the cylindrical excess introduced in Definition \ref{def:cilexc}. Their proofs are easy adaptations of those for the classical excess, 
see e.g.~\cite[Chapter 22]{maggi}. Note that, except for property 3), they hold also in the case $n=1$.

1) For all $0<r<s$ we have
\begin{equation}\label{eqeccraggi}
   \ex(E,r,\nu)\leq \left(\frac s r\right)^{2n+1}\ex(E,s,\nu).
\end{equation}

\medskip

2) If $(E_j)_{j\in\N}$ is a sequence of sets with locally finite $H$-perimeter
such that $E_j\to E$ as $j\to\infty $ in $L^1_{\mathrm{loc}}(\H^n)$, then we have
for any $r>0$
\begin{equation}\label{SCI-Exc}
\ex (E,r,\nu)\leq \liminf_{j\to\infty} \ex (E_j ,r,\nu).
\end{equation}

\medskip

3) Let $n\geq 2$. If $E\subset\H^n$ is a set such that
$\ex(E,r,\nu)=0$ and $0\in \hfr E$, then
\begin{equation}
   E\cap C_r =\big \{ p\in C_r : \h (p)<0\big\}.
\label{eccessozero}
\end{equation}
In particular, we have $\nu_E= \nu $ in $C_r\cap\partial E$. See also \cite[Proposition 3.6]{M1}.

\medskip

4) For any $\lambda>0$  and $r>0$ we have
\begin{equation}
  \label{exc-riscal}
          \ex(\lambda E,\lambda r,\nu) = \ex(E,r,\nu),
\end{equation} 
where $\lambda E = \{ (\lambda z,\lambda^2 t)\in\H^n: (z,t)\in E\}$.

\medskip

\medskip

\subsection{Proof of Theorem \ref{teoaltezza}} 
\label{pollo}

The following result is a first suboptimal version
of Theorem \ref{teoaltezza}.

\begin{lemma}\label{lemposition}
Let $n\geq 2$. For any $s\in(0,1)$, 
$\Lambda\in[0,\infty)$, and $r\in\:(0,\infty]$
with $\Lambda r\leq 1$, there exists a constant
$\omega(n,s,\Lambda,r)>0$ 
such that
if $E\subset \H^n$ is a $\Lrz$-minimum of $H$-perimeter in the cylinder
 $C_{2}$,
$0\in\partial E$, and $\ex(E,2,\nu)\leq \omega(n,s,\Lambda,r)$, then
\[
\begin{split}
& |\h(p)|<s\text{ for any }p\in\partial E\cap C_1,\\
& \Ln\big( \{p\in E\cap C_1 : \h(p)>s\}\big) =0,\\
& \Ln\big( \{p\in C_1\setminus E : \h(p)<-s\}\big) =0.
\end{split}
\]
\end{lemma}

\begin{proof}
By contradiction, assume that there exist $s\in(0,1)$ and
 a sequence of sets $(E_j)_{j\in\N}$  that are 
 $\Lrz$-minima in $C_{2}$ and such such that 
\[
\lim_{j\to\infty} \ex(E_j,2,\nu)=0
\]
and at least one of the following facts holds:
\begin{align}
\text{either } & \quad\text{there exists } 
p\in\partial E_j\cap C_1\text{ such
that }s\leq |\h(p)|\leq 1,\label{possibilitaA}\\
\text{or }& \quad\Ln\big( \{p\in E_j\cap C_1 : \h(p)>s\}\big) >
0,\label{possibilitaB}\\
\text{or }& \quad\Ln\big( \{p\in C_1\setminus E _j:
\h(p)<-s\}\big)>0.\label{possibilitaC}
\end{align}
By Theorem  \ref{teocompattezza}   in the Appendix A, 
there exists a measurable set $F\subset C_{5/3}$ such that
$F$ is a $\Lrz$-minimum in $C_{5/3}$, 
$0\in\partial F$ and (possibly up to
subsequences) $E_j\cap C_{5/3}\to F$ in $L^1(C_{5/3})$. 
By \eqref{SCI-Exc} and
\eqref{eqeccraggi}, we obtain
\[
\ex(F,4/3,\nu)\leq \liminf_{j\to\infty} \ex(E_j,4/3,\nu)
\leq \left(\tfrac 3
2\right)^{2n+1} \lim_{j\to\infty} \ex(E_j,2,\nu)=0.
\]
Since $0\in\partial F$, by \eqref{eccessozero} 
the set $F\cap C_{4/3}$ is (equivalent to) a halfspace with
horizontal inner normal $\nu=-X_1$, and, namely,
\[
       F\cap C_{4/3} = \{p\in C_{4/3}:\h(p)<0\}.
\]
Assume that \eqref{possibilitaA} holds for infinitely many $j$. 
Then, up to a
subsequence, there are   points $(p_j)_{j\in\N}$ and $p_0$ such that
\[
        p_j\in \partial E_j\cap C_1,\quad 
   |\h(p_j)|\in\: (s,1] \quad\text{and}\quad
     p_j\to p_0\in\partial F\cap \bar C_1.
\]
We  used again Theorem  \ref{teocompattezza} in the Appendix A. 
This is
a contradiction
because $\partial F\cap \bar C_1 
= \{p\in \bar C_1:\h(p)=0\}$. 
Here, we used $n\geq 2$.
Therefore, there exists $j_0\in\N$ such that
\[
       \{p\in\partial E_j\cap C_1:s\leq |\h(p)|\leq 1\}
   =\emptyset\quad\text{for all
}j\geq j_0,
\]
and hence
\[
    \mu_{E_j} (C_1\setminus\{p\in\Hn:|\h(p)|\leq s\})=0.
\]
This implies that, for $j\geq j_0$, $\chi_{E_j}$ is constant on the two
connected components  $C_1\cap \{p:\h(p) > s\}$ 
and $C_1\cap \{p:\h(p)< -
s\}$.
Since  the sequence $(E_j)_{j\in\N}$ converges in $L^1(C_1)$ 
to the halfspace $F$, then
for any $j\geq j_0$ we have 
\[
\begin{split}
& \chi_{E_j}=0\quad \Ln\text{-a.e. on }C_1\cap \{p:\h(p)> s\},
    \quad \textrm{and}\\
& \chi_{E_j}=1\quad \Ln\text{-a.e. on }C_1\cap \{p:\h(p)< - s\}.
\end{split}
\]
This contradicts both \eqref{possibilitaB} and
\eqref{possibilitaC} and concludes the proof.
\end{proof}

Let $\pi:\H^n\to\W$ be the group 
projection defined, for any $p\in\Hn$, by the formula
\[
            p = \pi(p)\ast (\h(p)\ee_1).
\] 
For any set $E\subset\H^n$  and 
for any $s\in\R$, we let $E^s = E \cap \H^n_s$ and 
we define the projection
\[
  E_s = \pi (E^s) = \big\{ w\in \W :  w\ast (s \ee_1) \in E\big\}.
\]
%

\begin{lemma}\label{lemmisuraeccesso}
Let $n\geq 2$,
let $E\subset\Hn$ be a set with locally finite $H$-perimeter 
and $0\in\partial
E$, and let $s_0\in(0,1)$ be such that
\begin{align}
& |\h(p)|<s_0\text{ for any }p\in\partial E\cap C_1,\label{sottot0}\\
& \Ln\big( \{p\in E\cap C_1 : \h(p)>s_0\}\big) =0,\label{sopranulla}\\
& \Ln\big( \{p\in C_1\setminus E : \h(p)<-s_0\}\big) =0.\label{sottotutto}
\end{align}
Then, for a.e.~$s\in(-1,1)$
 and for any continuous function
$\varphi\in C_c(D_1)$ we have, with   $M=\hfr E\cap C_1$
and $M_s = M\cap \{ \h>s\}$,
\begin{equation}
\int_{E_s\cap D_1}\varphi\,d\leb^{2n} 
 = -\int_{M_s}\varphi\circ\pi 
\, \langle \nu_E, X_1\rangle_g \,d\SQmu.
\label{444}
\end{equation}
In particular,
for
any Borel set $G\subset D_1$,
we have
\begin{align}
\leb^{2n}(G) = &-\int_{M\cap \pi^{-1}(G)}
\langle \nu_E, X_1\rangle_g \:d\SQmu, 
\label{222}
\\
\leb^{2n}(G)\leq &\:\SQmu(M\cap \pi^{-1}(G))\label{111}.
\end{align}
\end{lemma}

\begin{proof}
It is enough to prove \eqref{444}. 
Indeed, taking   $s<-s_0$ in \eqref{444} and recalling \eqref{sottot0} 
and \eqref{sottotutto}, we obtain 
\begin{equation}
   \int_{D_1}\varphi\:d\leb^{2n} =
-\int_M
  \varphi\circ\pi
\,\langle \nu_E, X_1\rangle_g  \:d\SQmu.
\label{333}
\end{equation}
Formula \eqref{222} follows from \eqref{333} 
by considering smooth
approximations of $\chi_G$. Formula \eqref{111} is immediate from \eqref{222}
and $|\langle \nu_E, X_1\rangle_g |\leq 1$.

We prove \eqref{444} for a.e.~$s\in(-1,1)$ and, namely, for those
$s$ satisfying the property \eqref{eqqfame2} below. 
Up to an approximation argument, we may assume that $\varphi\in
C^1_c(D_1)$. Let $r\in(0,1)$ and  $\s\in(\max\{s_0,s\},1)$ be
fixed.
We 
define 
\[
F=E
\cap (D_r\ast (s,\s))=E\cap \big\{ w\ast (\varrho \ee_1) \in \H^n
:w\in D_r,\, \varrho\in (s,\s)\big\}
. 
\]
We claim that
for a.e.~$r\in(0,1)$ and any $s$ satisfying  \eqref{eqqfame2} we have
\begin{equation}
\label{eqqclaim}
  \langle \nu_F, X_1\rangle_g 
  \mu_F = \langle \nu_E, X_1\rangle_g 
 \SQmu\res \hfr E\cap(D_r\ast(s,\s)) 
  + \leb^{2n} \res       E\cap D_r^ s .
\end{equation}
Above, we let $D_r^s  = \{ w\ast (s\ee_1) \in\H^n: w\in D_r\}$.
We postpone the proof of \eqref{eqqclaim}.
Let $Z$ be a   horizontal vector field of the form 
$Z = (\varphi\circ  \pi) X_1$.
We have $\divg Z=0$ because $X_1( \varphi\circ \pi)=0$.
Hence, we obtain 
\[
0=\int_F \divg\, Z\,d\Ln 
  = -\int_{\Hn}
\varphi\circ\pi \,\langle \nu_F, X_1\rangle_g  d\mu_F,
\]
i.e., by Fubini-Tonelli theorem and by \eqref{eqqclaim},
\[
-\int_{E_s\cap D_r} \varphi\,d\leb ^{2n}  
   =
-\int_{E\cap D_r^ s} \varphi\circ\pi\,d\leb ^{2n}  
   = \int_{\hfr
E\cap(D_r\ast(s,\s))} \varphi\circ\pi\,\langle \nu_E, X_1\rangle_g \,d\SQmu.
\]
Formula  \eqref{444} follows on letting first $r\nearrow 1$ and then 
$\s\nearrow 1$.

We are left with the proof of \eqref{eqqclaim}. 
Let $\psi\in C^1_c(\Hn)$ be a test function.
For any $w\in \W$ we let
\[
     E_w = \{\varrho \in\R:w \ast (\varrho \ee_1) \in E\},\quad 
   \psi_w(\varrho)=\psi(w\ast (\varrho \ee_1) ).
\]
Then we have $\psi_w\in C^1_c(\R)$ and, by Fubini-Tonelli theorem,
\begin{equation}\label{eqqfame1}
\begin{split}
       -\int_F X_1\psi&  \,d\Ln = -\int_{D_r}\int_s^\s 
      \chi_E(w\ast(\varrho \ee_1) )X_1\psi(w\ast (\varrho\ee_1))\,
        d\varrho \,d\leb^{2n}(w)\\
& = -\int_{D_r} \int_s^\s \chi_{E_w}(\varrho) \psi_w'(\varrho) \, d\varrho
  \,d\leb^{2n}(w)\\
& = \int_{D_r} \left[\int_s^\s \psi_w\,dD\chi_{E_w} -\psi_w(\s)\chi_{E_w}(\s^-)
+ \psi_w(s)\chi_{E_w}(s^+)\right]\,d\leb^{2n}(w),
\end{split}
\end{equation}
where $D\chi_{E_w}$ is
the derivative of $\chi_{E_w}$ in the sense of
distributions and $\chi_{E_w}(\s^-),\chi_{E_w}(s^+)$ 
are the classical trace
values of $\chi_{E_w}$ at the endpoints of the interval 
$(s,\s)$. We  used
the fact that the function $\chi_{E_w}$ 
is of bounded variation for
$\leb^{2n}$-a.e.~$w\in\W$, which in turn is a 
consequence of the fact that $X_1 \chi_E$ is
a signed Radon measure. For any such $w$, 
the trace of $\chi_{E_w}$ satisfies
\[
\chi_{E_w}(s^+) = \chi_{E_w}(s) = \chi_E(w\ast (s \ee_1))
 \quad \text{for a.e.~}s,
\]
so that, by Fubini's Theorem, for a.e.~$s\in\R$ there holds
\begin{equation}\label{eqqfame2}
\chi_{E_w}(s^+) = \chi_E(w\ast (s\ee_1))\quad 
\text{for }\leb^{2n}\text{-a.e.~}w\in D_1.
\end{equation}
With a similar argument, 
using \eqref{sopranulla} and the fact that $\s>s_0$ one
can see that
\begin{equation}\label{eqqfame3}
\chi_{E_w}(\s^-) = \chi_E(w\ast (\s \ee_1))=0\quad 
 \text{for }\leb^{2n}\text{-a.e. }w\in D_1.
\end{equation}
We refer the reader to \cite{AFP} for an 
extensive account on $BV$ functions and
traces. 
By \eqref{eqqfame1}, \eqref{eqqfame2} and \eqref{eqqfame3} 
we obtain
\[
\begin{split}
-\int_F X_1\psi\,d\Ln 
& = \int_{D_r} \int_s^\s \psi_w\,dD\chi_{E_w}\,d\leb^{2n}(w) +
\int_{D_r}\psi_w(s)\chi_{E_w}(s)\,d\leb^{2n}(w)\\
& = \int_{D_r\ast (s,\s)} \psi\, \langle \nu_E, X_1\rangle_g  d\mu_E 
 + \int_{E\cap D_r^s }
\psi\,d\leb^{2n}\\
& = \int_{\hfr E\cap(D_r\ast(s,\s))} \psi\,\langle \nu_E, X_1\rangle_g 
\,d\SQmu +
\int_{E\cap D_r^ s } \psi\,d\leb^{2n},
\end{split}
\]
and \eqref{eqqclaim} follows.
\end{proof}


\begin{corollary}\label{corperisoper}
Under the same assumptions and  notation
of Lemma \ref{lemmisuraeccesso}, 
for a.e.~$s\in(-1,1)$ there holds
\begin{equation}\label{eqwolenski}
0\leq \SQmu (M_s)-\leb^{2n} (E_s\cap D_1) \leq \ex(E,1,\nu).
\end{equation}
Moreover, we have 
\begin{equation}\label{excmeasexc}
\SQmu(M) -\leb^{2n}(D_1)=\ex(E,1,\nu).
\end{equation}
\end{corollary}

\begin{proof}
On approximating $\chi_{D_1}$ with functions $\varphi\in C_c(D_1)$, by
\eqref{444} we get
\[
\leb^{2n} (E_s\cap D_1) = - \int_{M_s} \langle  \nu_E, X_1\rangle_g\,
d\SQmu,
\]
and  the first inequality in \eqref{eqwolenski} follows. 
The second inequality 
follows from
\begin{equation}\label{eqferrante}
\begin{split}
 \SQmu (M_s)-\leb^{2n} (E_s\cap D_1)
& =  \int_{M_s} \big(1+\langle \nu_E,X_1\rangle_g \big)\,d\SQmu
 \\ 
& =  \int_{M_s} \frac{|\nu_E-\nu |_g^2}{2}\,d\SQmu 
 \\
&  \leq  
\ex(E,1,\nu).
\end{split}
\end{equation}
Notice that $\nu=-X_1$.
Finally, \eqref{excmeasexc} follows on choosing a suitable $s<-s_0$ and
recalling \eqref{sottot0} and \eqref{sottotutto}. In this case, 
 the
inequality in \eqref{eqferrante} becomes an 
equality and the proof is concluded.
\end{proof}


\begin{proof}[Proof of Theorem \ref{teoaltezza}]
{\em Step 1.} Up to replacing $E$ with the rescaled 
set 
$\lambda E = \{(\lambda z,\lambda^2 t)\in\H^n: (z,t)\in E\}$ with 
$\lambda = 1/{2k^2r}$
and recalling \eqref{exc-riscal}, 
we can without loss of generality assume
that $E$ is a $(\Lambda',\tfrac1{2k^2})$-minimum of $H$-perimeter in $C_2$
with
\begin{equation}\label{eqriscalati}
\frac{\Lambda'}{2k^2}\leq 1,\qquad 0\in\partial E,\qquad
\ex(E,2,\nu)\leq\e_0(n).
\end{equation}
Our goal is to find $\e_0(n)$ and $c_1(n)>0$ such that, if \eqref{eqriscalati} holds, then
\begin{equation}\label{eqast}
\sup\left\{ |\h(p)|: p\in \partial E\cap C_{1/2k^2}\right\} \leq c_1(n)\:
\ex(E,2,\nu)^{\tfrac{1}{2(2n+1)}}.
\end{equation}
We require 
\begin{equation}\label{eqep01}
\e_0(n)\leq \omega\big(n,\tfrac 1{4k},2k^2,\tfrac 1{2k^2}\big),
\end{equation}
where $\omega$ is given by Lemma \ref{lemposition}.
Two further assumptions on
$\e_0(n)$ will be made later in 
\eqref{eqsceltaep0} and \eqref{eqsceltaep0bis}.
By \eqref{eqriscalati}, $E$ is a $(2k^2,\tfrac1{2k^2})$-minimum in $C_2$.
Letting $M=\partial E\cap C_1$, by Lemma \ref{lemposition} and \eqref{eqep01} we have 
\begin{align}
& |\h(p)|< \tfrac1{4k}\text{ for any }p\in M\label{eqaltezza14},
\\
& \Ln\big( \{p\in E\cap C_1 : \h(p)>\tfrac1{4k}\}\big) =0,
\\
& \Ln\big( \{p\in C_1\setminus E : \h(p)<-\tfrac1{4k}\}\big) =0.
\end{align}
By \eqref{excmeasexc} and \eqref{eqeccraggi} we get
\begin{equation}\label{eqstep1A}
0\leq \SQmu(M)-\leb^{2n} (D_1)\leq \ex(E,1,\nu)\leq 2^{2n+1}\ex(E,2,\nu).
\end{equation}
Corollary \ref{corperisoper} implies that, for a.e.~$s\in(-1,1)$, 
\begin{equation}\label{eqstep1B}
0\leq \SQmu(M_s)-\leb^{2n} (E_s\cap D_1)\leq \ex(E,1,\nu)
\leq
2^{2n+1}\ex(E,2,\nu)
\end{equation}
where, as before, $M_s=M\cap\{\h>s\}$.

{\em Step 2.} Consider the function $f:(-1,1)\to[0,\SQmu(M)]$ defined by
\[
f(s)=\SQmu(M_s), \quad s\in(-1,1).
\]
The function $f$ is nonincreasing, 
right-continuous and, by \eqref{eqaltezza14}, it
satisfies
\begin{align*}
& f(s)=\SQmu(M)\text{ for any }s\in(-1,-\tfrac1{4k}],
\\
& f(s)=0\text{ for any }s\in(\tfrac1{4k},1].
\end{align*}
In particular, there exists 
$s_0\in(-\tfrac1{4k},\tfrac1{4k})$ such that
\begin{equation}\label{eqproprietat0}
\begin{split}
& f(s)\geq\SQmu(M)/2\text{ for any }s<s_0,
 \\
& f(s)\leq\SQmu(M)/2\text{ for any }s\geq s_0.
\end{split}
\end{equation}
Let $s_1\in(s_0,\tfrac1{4k})$ be such that
\begin{align}
& f(s)\geq\sqrt{\ex(E,2,\nu)}\text{ for any }s< s_1,
 \label{eqsceltat1}\\
& f(s)=\SQmu(M_s)\leq\sqrt{\ex(E,2,\nu)}
 \text{ for any }s\geq
s_1.\nonumber
\end{align}
We claim that there exists $c_2(n)>0$ such that
\begin{equation}\label{eqast3}
   \h(p) \leq s_1 + c_2(n) \ex(E,2,\nu)^{\tfrac1{2(2n+1)}}
 \quad\text{for any }p\in
    \partial E\cap C_{1/2k^2}.
\end{equation}
Inequality \eqref{eqast3} is trivial for any 
$p\in \partial E\cap C_{1/2k^2}$
with $\h(p)\leq s_1$. 
If $p\in \partial E\cap C_{1/2k^2}$ is such that $\h(p)>
s_1$, then
\begin{align*}
B_{{\text{\scriptsize\Fontamici h}}(p)-s_1}(p) \subset B_{1/2k}(p)\subset
B_{1/k}  \subset C_1.
\end{align*}
We  used the fact that $\| p\|_\KR \leq \tfrac1{2k}$ whenever $p\in
C_{1/2k^2}$, see \eqref{eqpallecilindri}. Therefore
\[
B_{{\text{\scriptsize\Fontamici h}}(p)-s_1}(p) \subset C_1 \cap\{\h>s_1\}
\]
and, by the density estimate \eqref{eqdensitaperimetro} 
of Theorem \ref{teostimedensita}
in Appendix A,
\begin{align*}
        k_3(n)(\h(p)-s_1)^{2n+1}& 
    \leq      \mu_E(B_{{\text{\scriptsize\Fontamici h}}(p)-s_1}(p))
    \leq      \mu_E(C_1\cap\{\h>s_1\})
 \\
         &=\SQmu(M_{s_1})=f(s_1)
    \leq \sqrt{\ex(E,2,\nu)}.
\end{align*}
This proves \eqref{eqast3}.

\medskip

{\em Step 3.} We claim that there exists $c_3(n)>0$ such that
\begin{equation}\label{eqast4}
s_1 -s_0 \leq c_3(n) \ex(E,2,\nu )^{\tfrac1{2(2n+1)}}.
\end{equation}
By the coarea formula \eqref{eqcoarea} with 
$h = \chi_{C_1}$, $D_1^s=  \{ p \in C_1 : \h(p)=s \}$,
and $E^s =\{ p \in E : \h(p)=s \}$,
we have 
\[
 \int_{-1}^1 \int_{D_1^s}d\mu_{E^s}^s\,ds 
   =
   \int_{C_1}\sqrt{1-\langle \nu_E,X_1\rangle_g^2}\,d\mu_E 
\leq 
    \sqrt 2\int_M \sqrt{1+\langle \nu_E,X_1\rangle_g }\,d\SQmu.
\]
By H\"older inequality,
\eqref{eqdensitaperimetro},
%
\eqref {222}, and \eqref{eqstep1A}, we deduce that
\begin{equation}
 \label{eqrunggaldier}
\begin{split}
    \int_{-1}^1 \int_{D_1^s}d\mu_{E^s}^s\,ds 
   \leq  & \sqrt{2\SQmu(M)}
  \left( \int_M (1+\langle \nu_E,X_1\rangle_g)\,d\SQmu\right)^{1/2}
   \\ 
\leq & c_4(n) (\SQmu(M)- \leb^{2n}(D_1))^{1/2}
   \\
   \leq & c_5(n) \sqrt{\ex(E,2,\nu)}.
\end{split}
\end{equation}
By Corollary
\ref{corperisoper} and \eqref{eqstep1A}, 
we obtain, for a.e.~$s\in[s_0,s_1)$, 
\begin{equation}\label{sput}
\begin{split}
      \leb^{2n} (E_s\cap D_1) & \leq \SQmu (M_s) 
 =  f(s)\leq f(s_0)
     \\
    & \leq  \frac{\SQmu(M)}2 \leq  
  \frac{\leb^{2n} (D_1)+2^{2n+1}\ex(E,2,\nu)}{2}\leq\frac34
\leb^{2n} (D_1).
\end{split}
\end{equation}
The last inequality holds provided that  
\begin{equation}\label{eqsceltaep0}
2^{2n+1}\e_0(n)\leq \frac{\leb^{2n} (D_1)}4.
\end{equation}

Let $\Omega_s = (-s\ee_1)\ast D_1^s =(-s\ee_1)\ast D_1 \ast(s\ee_1)$
and $F_s = (-s\ee_1)\ast E^s$. 
We have
\begin{equation} \label{FF11}
\leb^{2n}(\Omega_s) = \leb^{2n}(D_1^s) = \leb^{2n}(D_1),
\end{equation}
and, by \eqref{sput},
\begin{equation} \label{FF22}
\leb^{2n}(F_s\cap \Omega_s) = \leb^{2n}(E^s\cap D_1^s) =
 \leb^{2n}(E_s\cap D_1)\leq \frac 34 \leb^{2n} (D_1).
\end{equation}
Moreover, by left invariance we also have
\begin{equation} \label{FF33}
 \mu_{E^s}^s(D_1^s) = \mu_{F_s}^0(\Omega_s).
\end{equation}

By \eqref{FF11}--\eqref{FF33} and  
Corollary  \ref{teoisoprelW}, there exists a constant $k(n)>0$
independent of $s\in (-1,1)$
such that 
\begin{equation}\label{FF44}
\mu_{E^s}(D_1^s ) =\mu_{F_s}^0(\Omega_s) 
\geq k(n)\, \leb^{2n} (F_s\cap
\Omega_s)^{\tfrac{2n}{2n+1}}
=
k(n)\, \leb^{2n} (E^s\cap
D_1^s)^{\tfrac{2n}{2n+1}}.
\end{equation}
This, together with \eqref{eqrunggaldier}, gives
\begin{align*}
c_6(n) \sqrt{\ex(E,2,\nu)}  
\stackrel{\hphantom{\eqref{eqsceltat1}}}{\geq} 
 &
  \int_{s_0}^{s_1} \leb^{2n} (E^s\cap D_1^s)^{\tfrac{2n}{2n+1}}\,ds
  \\
\stackrel{\eqref{eqstep1B}}{\geq} & 
  \int_{s_0}^{s_1}
\Big(\SQmu(M_s) -2^{2n+1}\ex(E,2,\nu))\Big)
    ^{\tfrac{2n}{2n+1}}\,ds
\\
\stackrel{\eqref{eqsceltat1}}{\geq} 
& 
\int_{s_0}^{s_1}
\left(\sqrt{\ex(E,2,\nu)}-2^{2n+1}\ex(E,2,\nu))\right)^{\tfrac{2n}{2n+1}}
\,ds\\
\stackrel{\hphantom{\eqref{eqsceltat1}}}{\geq} 
& 
\frac 12 \int_{s_0}^{s_1}
{\ex(E,2,\nu)}^{\tfrac{n}{2n+1}}\,ds.
\end{align*}
In the last inequality, we  require that $\e_0(n)$ satisfies
\begin{equation}
\label{eqsceltaep0bis}
\sqrt z -2^{2n+1}z \geq \tfrac 12 \sqrt z\quad \textrm{for all }
  z\in[0,\e_0(n)].
\end{equation}
It follows that
\[
c_6(n) \sqrt{\ex(E,2,\nu)} \geq \frac 12\:\ex(E,2,\nu)^{\tfrac{n}{2n+1}}
(s_1-s_0),
\]
and \eqref{eqast4} follows.

\medskip

{\em Step 4.} Recalling \eqref{eqast3} and \eqref{eqast4}, 
we proved that
there exist $\e_0(n)$ and $c_6(n)$ such that the following holds. 
If $E$ is a
$(2k^2,\tfrac1{2k^2})$-minimum of $H$-perimeter 
in $C_2$ such that
\[
0\in\partial E,\qquad \ex(E,2,\nu)\leq\e_0(n)
\]
and $s_0=s_0(E)$ satisfies \eqref{eqproprietat0}, then 
\begin{equation}\label{eqast2}
\h(p)-s_0 \leq c_7(n) \ex(E,2,\nu)^{\tfrac1{2(2n+1)}}\quad
 \text{for any }p\in\partial
E\cap C_{1/2k^2}.
\end{equation}

Let us introduce the mapping  $\Psi:\Hn\to\Hn$ 
\[
\Psi(x_1,x_2\dots,x_n,y_1,\dots,y_n,t)=(-x_1,-x_2,\dots,-x_n,y_1,\dots,y_n,-t)\,
.
\]
Then we have  $\Psi^{-1}=\Psi$, $\Psi(C_2)=C_2$, 
$\langle X_j ,\nu _{\Psi(F)}\rangle_g = -\langle X_j ,\nu _{F}\rangle_g 
\circ\Psi$,
$\langle Y_j ,\nu _{\Psi(F)}\rangle_g = \langle Y_j ,\nu _{F}\rangle_g 
\circ\Psi$, and $ \mu_{\Psi(F)} = \Psi_{\#}\mu_F$,  
for any set $F$ with locally finite $H$-perimeter; 
here, $\Psi_\#$ denotes the
standard push-forward of measures. 
Therefore, the set $\wt E=\Psi(\Hn\setminus E)$
satisfies the following properties:
\begin{itemize}
\item[i)] $\wt E$ is a $(2k^2,\tfrac1{2k^2})$-minimum 
of $H$-perimeter in $C_2$;
\item[ii)] $0\in\partial \wt E$ and
\[
\ex(\wt E,2,\nu) = \frac{1}{2^{Q}}\int_{\hfr \wt E\,\cap C_2}
|\nu_{\wt
E}-\nu| _g ^2 \,d\SQmu =
\ex( E,2,\nu )  \leq\e_0(n) ;
\]
\item[iii)] setting $\wt M=\hfr \wt E\,\cap C_1=\Psi(M)$ and 
$\wt f(s)=\SQmu(\wt
M\,\cap\{\h>s\})$,
we have 
\[
\begin{split}
& \wt f(s)\geq\SQmu(\wt M)/2=\SQmu(M)/2\text{ for any }s<-s_0,\\
& \wt f(s)\leq\SQmu( M)/2\text{ for any }s\geq -s_0.
\end{split}
\]
\end{itemize}
Formula \eqref{eqast2} for the set $\wt E$ gives 
\[
\h(p)+s_0 \leq c_7(n) \ex( E,2,\nu)^{\tfrac1{2(2n+1)}}\quad\text{for any
}p\in\partial \wt E\cap C_{1/2k^2}.
\]
Notice that we have $p\in\partial \wt E$ if and only if 
$\Psi(p)\in\partial  E$ and, moreover,  $\h(\Psi(p))=-\h(p)$.
Hence, we have
\begin{equation}\label{eqast2tilde}
-\h(p)+s_0 \leq c_7(n) \ex(E,2,\nu )^{\tfrac1{2(2n+1)}}\quad\text{for any
}p\in\partial E\cap C_{1/2k^2}.
\end{equation}
By \eqref{eqast2} and \eqref{eqast2tilde} we obtain
\begin{equation}\label{eqancoradue}
|\h(p)-s_0| \leq c_7(n) \ex(E,2,\nu)^{\tfrac1{2(2n+1)}}\quad\text{for any
}p\in\partial E\cap C_{1/2k^2},
\end{equation}
and, in particular,
\begin{equation}\label{eqancorauna}
|s_0|\leq c_7(n) \ex(E,2,\nu)^{\tfrac1{2(2n+1)}},
\end{equation}
because $0 \in\partial E\cap C_{1/2k^2}$. 
Inequalities \eqref{eqancoradue} and
\eqref{eqancorauna} give \eqref{eqast}. This completes  the proof.
\end{proof}

\section{Appendix A} 

We list some basic properties of $\Lambda$-minima of $H$-perimeter in $\H^n$.
The proofs are straightforward adaptations of the proofs for $\Lambda$-minima
of  perimeter in $\R^n$.

\begin{theorem}[Density estimates]\label{teostimedensita}
There exist  constants $k_1(n), k_2(n), k_3(n), k_4(n)>0$ with the
following property. If $E$ is a $\Lrz$-minimum of  $H$-perimeter in
$\Omega\subset \H^n$, 
$p\in\partial E\cap\Omega$, $B_r(p)\subset\Omega$ and $s<r$, 
then
\begin{align}
& k_1(n)\leq \frac{\Ln(E\cap B_s(p))}{s^{2n+2}}\leq
k_2(n)\label{eqdensitavolume}
\\
& k_3(n)\leq \frac{\mu_E(B_s(p))}{s^{2n+1}}\leq
k_4(n)\label{eqdensitaperimetro}.
\end{align}
\end{theorem}

For a proof see \cite[Theorem 21.11]{maggi}. By standard arguments Theorem
\ref{teostimedensita} implies the following corollary.

\begin{corollary}
If $E$ is a $\Lrz$-minimum of   $H$-perimeter in $\Omega\subset\H^n$, 
then  
\[
\SQmu((\partial E\setminus \hfr E)\cap\Omega)=0.
\]
\end{corollary}

\begin{theorem}\label{teocompattezza}
Let $(E_j)_{j\in\N}$ be a sequence of $\Lrz$-minima of  
$H$-perimeter in an open
set $\Omega\subset\H^n$, $\Lambda r\leq 1$. Then there exists a $\Lrz$-minimum
$E$ of $H$-perimeter  in $\Omega$ and a subsequence $(E_{j_k})_{k\in\N}$ 
such that
\[
  E_{j_k} \to E\quad \text{in $L^1_{\mathrm{loc}}(\Omega)$}
\quad\text{and}\quad 
\nu_{E_{j_k}}\, 
\mu_{E_{j_k}} 
\stackrel{\ast}{\rightharpoonup} \nu_E \mu_E
\]
as $k\to\infty$. Moreover, the measure theoretic boundaries  
$\partial E_{j_k}$ converge to $\partial E$ in the sense of Kuratowski, i.e.,
\begin{itemize}
\item[i)] if $p_{j_k}\in \partial E_j\cap \Omega $ and $p_{j_k}\to p\in \Omega$, then
$p\in\partial E$;
\item[ii)] if $p\in\partial E\cap \Omega $, then there exists 
a sequence $(p_{j_k})_{k\in\N}$ such that
$p_{j_k}\in \partial E_{j_k}\cap \Omega $ and $p_{j_k}\to p$.
\end{itemize}
\end{theorem}

For a proof in the case of the perimeter in $\R^n$,
see \cite[Chapter 21]{maggi}.



\section{Appendix B}
We define a Borel unit normal $\nu_R$ to an 
$\SQmu$-rectifiable set $R\subset\H^n$ and we show that 
the definition is well posed $\SQmu$-almost everywhere, up to the sign.
The normal $\nu_S$ to an $H$-regular hypersurface $S\subset\H^n$ is defined in 
\eqref{nuesse}.

\begin{definition}\label{def:normrett}
Let $R\subset\H^n$ be an $\SQmu$-rectifiable set such that 
\begin{equation}\label{eq:tiziospagnolo}
 \mathcal S^{2n+1} \Big (R\setminus \bigcup_{j\in\N} S_j\Big) = 0
\end{equation}
for a sequence of $H$-regular hypersurfaces $(S_j)_{j\in\N}$ in $\H^n$. 
For any   $p\in R\cap \bigcup_{j\in\N} S_j$ we define
\[
\nu_R(p)=\nu_{S_{\bar\jmath}}(p),
\]
where $\bar\jmath$ is the unique integer 
such that $p\in S_{\bar\jmath}\setminus \bigcup_{j<\bar\jmath} S_j$.
\end{definition}

We show that Definition   \ref{def:normrett} is well posed, up to a sign, 
for $\SQmu$-a.e.~$p$. Namely, let  
$(S_j^1)_{j\in\N}$ and $(S_j^2)_{j\in\N}$ 
be two sequences of $H$-regular hypersurfaces  
in $\H^n$ for which \eqref{eq:tiziospagnolo} holds
and denote by  $\nu^1_R$ and $\nu_R^2$, respectively, the associated normals to $R$ according to Definition \ref{def:normrett}.
We show that $\nu_R^1=\nu_R^2$ $\SQmu$-a.e.~on $R$, up the the sign.

Let $A\subset R$ be the set of points such that 
either $\nu^1_R(p)$   is not defined, or  
$\nu^2_R(p)$  is not defined, or
they are both defined and 
$\nu^1_R(p)\neq\pm\nu_R^2(p)$.
It is enough to show that $\SQmu(A)=0$. 
This is a consequence of the following lemma.

\begin{lemma}
Let $S_1,S_2$ be two $H$-regular hypersurfaces  in $\H^n$ and let
\[
A=\{p\in S_1\cap S_2:\nu_{S_1}(p)\neq\pm\nu_{S_2}(p)\}.
\]
Then, the Hausdorff dimension of $A$ in the Carnot-Carath\'eodory metric is at
most $2n$, $\mathrm{dim}_{CC}(A)\leq 2n$,
and, in particular, $\SQmu(A)=0$.
\end{lemma}

\begin{proof}
 The blow-up of $S_i$, $i=1,2$,
at a point $p\in A$ is a 
vertical hyperplane $\Pi_i\times\R\subset\R^{2n}\times\R\equiv\H^n$,
see e.g.~\cite{FSSCMathAnn},
where:
\begin{itemize}
\item[i)] by blow-up of $S_i$ at $p$ we mean the limit
\[
\lim_{\lambda\to \infty} \lambda(p^{-1}\ast S_i)
\]
in the Gromov-Hausdorff sense.
Recall that, for $E\subset\H^n$, we define 
$\lambda E = \{ (\lambda z,\lambda^2 t)\in\H^n: (z,t)\in E\}$).

\item[ii)] For $i=1,2$, $\Pi_i\subset\R^{2n}$ is the normal 
hyperplane to $\nu_{S_i}(p)\in H_p\equiv\R^{2n}$.
\end{itemize}

\noindent It follows that the blow-up of $A$ at $p$ is 
contained in the blow-up of $S_1\cap S_2$ at $p$, i.e., 
in $(\Pi_1\cap\Pi_2)\times \R$. 
Since $\nu_{S_1}(p)\neq\pm\nu_{S_2}(p)$, $\Pi_1\cap\Pi_2$ 
is a $(2n-2)$-dimensional plane in $\R^{2n}$, 
and we  conclude thanks to the following lemma.
\end{proof}

\begin{lemma}\label{lem:dimblowup}
Let $k=0,1,\ldots, 2n$ and  $A\subset\H^n$ be such that 
for any $p\in A$, the blow-up of $A$ at $p$ is contained 
in $\Pi_p\times\R$ for some plane $\Pi_p\subset\R^{2n}$ of dimension $k$. 
Then we have   $\mathrm{dim}_{CC}(A)\leq k+2$.
\end{lemma}

\begin{proof}
We claim that for any $\eta>0$ we have  
\begin{equation}\label{eq:tesi1}
\mathcal S^{k+2+\eta}(A)=0.
\end{equation}
Let $\ep\in(0,1/2)$ be  such  
that $C\ep^{\eta}\leq 1/2$, where $C=C(n)$ 
is a constant that will be fixed later in the proof. 
By the definition of blow-up, 
for any $p\in A$ there exists $r_p>0$ such that
for all $r\in(0,r_p)$ we have
\[
(p^{-1}\ast A) \cap U_r\subset (\Pi_p)_{\ep r}\times\R,
\]
where $(\Pi_p)_{\ep r}$ denotes 
the $(\ep r)$-neighbourhood of $\Pi_p$ in $\R^{2n}$. 
For any $j\in\N$ set
\[
A_j=\{p\in A\cap B_j:r_p>1/j\}.
\]
To prove \eqref{eq:tesi1}, it is enough to prove that
\[
\mathcal S^{k+2+\eta}(A_j)=0
\]
for any fixed $j\geq 1$. This, in turn, will follow if we show that, for any fixed $\delta\in(0,\tfrac1{2j})$, one has
\begin{equation}\label{eq:tesi2}
\begin{split}
\inf \Big\{ \sum_{i\in\N} r_i^{k+2+\eta}  &  : A_j\subset \bigcup_{i\in\N}  
U_{r_i}(p_i), \, r_i <2\ep\delta \Big\}\leq
  \\
\leq\: & \frac 12
\inf \Big\{ \sum_{i\in\N} r_i^{k+2+\eta}   : A_j\subset \bigcup_{i\in\N}  
U_{r_i}(p_i), \, r_i <\delta \Big\}.
\end{split}
\end{equation}
Let  $(U_{r_i}(p_i))_{i\in \N}$ be a covering of $A_j$ 
with balls of radius smaller than $\delta$.
There exist points   $\bar p_i\in A_j$ 
such that $(U_{2r_i}(\bar p_i))_{i\in \N}$ 
is a covering of $A_j$ with balls of radius smaller 
than $2\delta<1/j$. By definition of $A_j$, we have 
\[
(\bar p_i^{-1}\ast A_j)\cap U_{2r_i}
\subset ((\Pi_{\bar p_i})_{\ep r_i}\times\R)\cap U_{2r_i}.
\]
The set $((\Pi_{\bar p_i})_{\ep r_i}\times\R)\cap U_{2r_i}$ 
can be covered by a family of balls $(U_{\ep r_i}(p_h^i))_{h\in H_i}$ 
of radius $\ep r_i<2\ep\delta$ in such a way that the cardinality of $H_i$
is bounded by $C\ep^{-k-2}$, where  
the constant $C$ depends only on $n$ 
and not on $\ep$. In particular, 
the family of balls $( U_{\ep r_i}(\bar p_i\ast p_h^i))_{i\in \N, h\in H_i}$  
is a covering of $A_j$ and
\[
\begin{split}
\sum_{i\in\N} \sum_{ h\in H_i} (\mathrm{radius}\:
U_{\ep r_i}(\bar p_i\ast p_h^i))^{k+2+\eta} 
& = \sum_{i\in \N} \sum_{h\in H_i} (\ep r_i)^{k+2+\eta}  
\leq C\ep^{-k-2} \sum_{i\in \N} (\ep r_i)^{k+2+\eta}\\
&= C\ep^{\eta} \sum_{i\in\N} r_i^{k+2+\eta} 
\leq \frac 12 \sum_{i} r_i^{k+2+\eta}.
\end{split}
\]
This proves \eqref{eq:tesi2} and concludes the proof.
\end{proof}

\bibliographystyle{acm}

\end{document}